\documentclass[11pt,reqno, a4paper]{article}

\usepackage[LCY,T1]{fontenc}
\usepackage[english]{babel}
\usepackage[utf8]{inputenc}

\usepackage{amsmath, amsthm, amssymb}
\usepackage{mathtools}
\usepackage{thmtools}
\usepackage{comment}
\usepackage{enumitem}
\usepackage{booktabs}
\usepackage{xfrac}
\usepackage[a4paper, left=2.5cm, right=2.5cm, top=2.5cm, bottom=3cm, footskip=1.6cm]{geometry}
\numberwithin{equation}{section}

\usepackage{hyperref}

\newtheorem{theorem}{Theorem}[section]
\newtheorem{lemma}[theorem]{Lemma}
\newtheorem{corollary}[theorem]{Corollary}
\newtheorem{proposition}[theorem]{Proposition}
\newtheorem{remark}[theorem]{Remark}

\newcommand{\N}{\mathbb{N}}

\newcommand{\Q}{\mathbb{Q}}
\newcommand{\R}{\mathbb{R}}

\newcommand{\E}{\mathbb{E}}
\newcommand{\T}{\mathbb{T}}
\newcommand{\Id}{\operatorname{Id}}
\newcommand{\Prob}{\mathbb{P}}

\newcommand{\ee}{\mathrm{e}}

\DeclareMathOperator{\tr}{tr}

\newcommand{\dint}{\,\mathrm{d}}

\newcommand{\cH}{\mathcal{H}}

\DeclareMathOperator{\rank}{rank}
\DeclareMathOperator{\spann}{span}

\DeclareMathOperator*{\bigtimes}{\text{\Large$\times$}}



\begin{document}
	%%%%%%%%%%%%%%%%%%%%%%%%%%%%%%%%%%%%%%%%%%%%%%%%%%%%%%%%%%%%%%%%%%%%%%%%%%%%%%%
	
	\title{Random nodes with positive weights: Richter–Tchakaloff rules and optimal
		positive quadrature}

	\author{Sebastian Neumayer,
		%$\!{}^{a}$,
		Martin Sch\"afer\footnote{Corresponding author, Email:
			martin.schaefer@math.tu-chemnitz.de},
		%$\!{}^{a}$,
		Tino Ullrich
		%$\!{}^{a,}$$\!{}^{a,}$\footnote{Corresponding author, Email:tino.ullrich@math.tu-chemnitz.de}
		\\\\
		%${}^{a}\!$ 
		Chemnitz University of Technology, Faculty of Mathematics\\[2mm]
	}
	%%%%%%%%%%%%%%%%%%%%%%%%%%%%%%%%%%%%%%%%%%%%%%%%%%%%%%%%%%%%%%%%%%%%%%%%%%%%%%%%
	%%%%%
	%%%%%%%%%%%%%%%%%%%%%%%%%%%%%%%%%%%%%%%%%%%%%%%%%%%%%%%%%%%%%%%%%%%%%%%%%%%%%%%%
	%%%%%
	
	\date{\today}
	
	\maketitle
	
	\begin{abstract}
		\noindent
		We show that independent random points support exact quadrature rules with positive weights.
		For any $n$-dimensional space $V$ of integrable functions a constant multiple of $n\log n$ points sampled from a suitable density yields, with high probability, positive weights which integrate the functions in the space exactly.
		The weights stay close to the natural importance sampling weights according to the Lewis density with respect to $V$. They are the result of a weight adaptation procedure which requires an explicit  $L_1$-norm discretization result that is proved on the way. As an application, we study numerical integration in reproducing kernel Hilbert spaces.
		Combining the random rules with a compression/subsampling step, we show that positive quadrature rules with constantly oversampled nodes achieve the best possible worst-case error bound in terms of the singular numbers of the embedding into $L_2$.
	\end{abstract}
	

	\section{Introduction}
	%\label{sec:intro}
	Numerical integration is one of the most basic tasks in numerical analysis. Quadrature rules with non-negative weights are the preferred choice in many applications.
	They provide numerical stability and preserve positivity of the integrand. Furthermore, if $\mu$ is a probability measure and the rule integrates constants exactly, their weights sum up to one, i.e., the rule itself is a discrete probability measure approximating~$\mu$.
	
	A classical design principle for quadrature rules is exactness on a finite-dimensional space $V$, as for Gaussian rules and algebraic polynomials, spherical designs and spherical polynomials, or rules for trigonometric polynomials with frequencies in a hyperbolic cross.
	Such a rule inherits good error bounds for all functions that are well approximated from $V$, for example, provided that its weights are under control.
	Richter's theorem, see~\cite{Richter57}, which is a generalization of Tchakaloff's theorem, see~\cite{Tch57}, guarantees that positive exact rules exist, but its proof is non-constructive, and it says nothing about the error outside of $V$.
	
	In this paper, we provide a construction.
	First, we show that i.i.d.\ random nodes, drawn from a suitable density, lead to positive quadrature rules that are exact on a given $n$-dimensional space $V$ as soon as their number of nodes exceeds a constant multiple of $n\log n$. We combine a duality argument going back to Mhaskar, Narcowich and Ward \cite{MhNaWa01}
with a new and sharp general $L_1$-norm discretization result for i.i.d.\ random nodes, see Remark~\ref{rem:discretization}, which
simultaneously controls the quadrature error.
This allows us to adapt the importance sampling weights such that the resulting rule is
exact on $V$ while its weights stay within a box of fixed relative size.

	Second, we use these rules obtained from a random draw for worst-case integration in reproducing kernel Hilbert spaces and show that positive rules achieve the universally optimal error bound in terms of the singular numbers of the embedding into $L_2$ after a further deterministic compression/subsampling applied to the initial node set. In this universal sense, the restriction to positive weights does not cost anything.
	
	\subsection{Main results}
	
	\paragraph{Random Richter-Tchakaloff rules.}
	Let $(\Omega,\Sigma,\mu)$ be a measure space and $V\subset L_1(\Omega,\mu)$ a space of real-valued functions with $\dim V=n$.
	Throughout, given nodes $X\coloneqq (x_1,\ldots,x_m)\in\Omega^m$ and non-negative weights $\mathbf{w}\coloneqq (w_1,\ldots,w_m)$, we write for $f\in V$
	\begin{align}
		\label{eqdef:Q_{mathbf{w}}(X)}
		Q_{\mathbf{w}}(X)f \coloneqq  \sum_{i=1}^{m} w_i f(x_i).
	\end{align}	
	The positive and exact quadrature results in the literature are mostly non-constructive. Hence, they do not provide a practical method for finding exact positive quadrature rules for $V$.
	Due to the prominence of Richter's theorem and Tchakaloff's theorem in the literature, we will henceforth call those quadrature rules Richter-Tchakaloff rules for $V$.
	A notable example with additional structure are spherical $t$-designs, i.e., rules with equal weights on the sphere $\mathbb{S}^d$ which are exact for all spherical polynomials of degree at most $t$.
	By a fundamental result of Bondarenko, Radchenko and Viazovska~\cite{BoRaVi13}, they exist with $\asymp_d t^d$ nodes, which is the optimal order, but again non-constructively.
    
	Our first result is a simple random construction of Richter-Tchakaloff rules.
	It is based on the $L_1$-Lewis density $\varrho_{\rm L}$ of $V$, see~\cite[Thm.\ 1]{Lewis78} and Proposition~\ref{prop:lewis} below, a probability density which plays the role of the Christoffel density $\varrho_{\rm C}$, see Remark~\ref{rem:2.3}, in $L_1$.
	
	\begin{theorem}
		\label{thm:intro_RT}
		Let $V\subset L_1(\Omega,\mu)$ be a space of real-valued functions with $\dim V=n$, let $\varrho$ be a probability density with $\varrho_{\rm L}\le\gamma\varrho$ for some $\gamma\ge1$, and let $\delta\in(0,1)$.
		If
		\begin{equation}
			\label{eq:number_of_random_nodes}
			m\ge 134 \gamma\, n\log\Bigl(\frac{4n}{\delta}\Bigr)
		\end{equation}
		nodes $Y=(y_1,\ldots,y_m)$ are drawn independently according to $\varrho\mu$, then, with probability at least $1-\delta$, there exist weights $\mathbf w=(w_1,\ldots,w_m)$ with
		\begin{equation*}
			%\label{eq:intro_weights}
			\frac{1}{2m\varrho(y_i)}\le w_i\le\frac{3}{2m\varrho(y_i)},\qquad i=1,\ldots,m,
		\end{equation*}
		such that
		\begin{equation*}
			%\label{eq:intro_exact}
			\sum_{i=1}^m w_i f(y_i)=\int_\Omega f\dint\mu\qquad\text{for all }f\in V.
		\end{equation*}
		If, in addition, $\mathbf{1}_\Omega\in V$ and $\mu(\Omega)=1$, then $\sum_{i=1}^m w_i=1$.
	\end{theorem}

	\noindent
    Theorem~\ref{thm:intro_RT} is Corollary~\ref{cor:Prob_Tchak} in Section~\ref{sec:RandomRules}.
	The weights are obtained by adapting the importance sampling weights $w_i^0=1/(m\varrho(y_i))$ within the box $[\frac12w_i^0,\frac32w_i^0]$, i.e., by solving a linear feasibility problem.
	
	\begin{remark}
		\label{rem:intro_RT}
		In Theorem~\ref{thm:intro_RT} sampling need not be performed strictly according to the Lewis density $\varrho_{\rm L}$. Any sampling density $\varrho$ satisfying $\varrho_{\rm L}\le\gamma\varrho$ for some $\gamma\ge1$ is suitable. In particular, if $\mu(\Omega)<\infty$ and $\varrho_{\rm L}$ is bounded, one may sample uniformly with respect to the underlying measure $\mu$, provided the required number of nodes is increased by a factor of $\gamma$, see~\eqref{eq:number_of_random_nodes}.
	\end{remark}
	
	Let us mention some consequences.
	\begin{itemize}
		\item {\em Almost equal weights.} For $\tau\in(0,\frac12]$, the box can be shrunk to $[(1-\tau)w_i^0,(1+\tau)w_i^0]$ at the price of $m\ge34\tau^{-2}\gamma n\log(4n/\delta)$ nodes (Corollary~\ref{cor:near_equal}). Thus, if $\varrho$ is chosen constant, which by Remark~\ref{rem:intro_RT} is possible whenever $\mu(\Omega)<\infty$ and the Lewis density is bounded, one obtains almost equal weights.
		\item {\em ``Almost'' spherical designs.}
		If $V$ and $\mu$ are invariant under a transitive group of measure preserving transformations, the Lewis density $\varrho_{\rm L}$ is constant itself, and the nodes can be drawn uniformly with $\gamma=1$ in~\eqref{eq:number_of_random_nodes} (Remark~\ref{rem:sphere}).
		For the space $\Pi_t$ of spherical polynomials of degree at most $t$ on $\mathbb{S}^d$ with $n=\dim\Pi_t\asymp_d t^d$, this shows that $m\asymp_d\tau^{-2}t^d\log(t/\delta)$ independent uniformly distributed points on $\mathbb{S}^d$ admit, with probability at least $1-\delta$, weights with
		\begin{equation}
			\label{eq:intro_almost_design}
			\frac{1-\tau}{m}\le w_i\le\frac{1+\tau}{m},\qquad \sum_{i=1}^m w_i=1,
		\end{equation}
		which integrate $\Pi_t$ exactly.
		Thus, random points are almost spherical $t$-designs with high probability, and the corresponding weights are certainly computable.
		The same holds, e.g., for trigonometric polynomials with frequencies in a hyperbolic cross on $\T^d$.
        \item {\em Richter-Tchakaloff rules with bounded weights.} Theorem~\ref{thm:RT_bounded} shows how a probabilistic node generation as in Theorem~\ref{thm:intro_RT} combined with a subsampling procedure, established by Lemma~\ref{lem:compress}, yields Richter-Tchakaloff rules of sample size $2n$ for any given $n$-dimensional space $V\subset L_1(\Omega,\mu)$ with weights $0\le w_i\lesssim 1/n$.
	\end{itemize}
	
	\paragraph{Positive quadrature in reproducing kernel Hilbert spaces.}
	Let now $\cH=\cH_K$ be a reproducing kernel Hilbert space (RKHS) of real-valued measurable functions on $(\Omega,\Sigma)$.
	The associated symmetric positive semidefinite kernel $K\colon \Omega\times\Omega\to\R$ is assumed to be $(\Sigma\otimes\Sigma)$-measurable, and $\mu$ is a positive measure on $(\Omega,\Sigma)$ with $\mu(\Omega)<\infty$ such that the trace of $K$ is finite, i.e.,
	\begin{equation}
    \label{P2:integrab}
		\tr(K)\coloneqq \int_{\Omega} K(x,x)\dint\mu(x) < \infty.
	\end{equation}
	These assumptions are a standard scenario in the literature on sampling and integration in RKHS, see, e.g.,~\cite{NoWo10,KrUl21,KaUlVo21,NaSchUl22, DoKrUl23}.
	Under condition~\eqref{P2:integrab}, $\cH$ embeds into $L_2(\Omega,\mu)$, and the embedding $\Id_{K,\mu}\colon  \cH \to L_2(\Omega,\mu)$ is Hilbert-Schmidt, see~\cite{HeBo04,StSc12}. Since $L_2(\Omega,\mu)$ consists of equivalence classes, it is not necessarily injective.
	If $\cH$ is separable, we denote by $\sigma_1\ge\sigma_2\ge\ldots\ge0$ the, at most countable, system of nontrivial singular numbers of $\Id_{K,\mu}$, so that $\tr(K)=\sum_{j\ge1}\sigma_j^2$, and by $V_n$ the span of the first $n$ right singular functions, see Section~\ref{sec:RKHS}.
	
	Our second result applies Theorem~\ref{thm:intro_RT} to the space $V_n^\prime\coloneqq V_n + \spann\{\mathbf{1}_\Omega\}$, with nodes drawn from the mixture $\varrho_n$ of the uniform density, the Lewis density of $V_n^\prime$, and the normalized tail of the kernel diagonal, see~\eqref{mix_density}.
	
	\begin{theorem}
		\label{thm:intro_RKHS}
		Let $(\Omega,\Sigma,\mu)$ be a probability space, let $\cH$ be a separable RKHS of real-valued functions on $\Omega$ with finite trace, and let $n\in\N$ and $\delta\in(0,1)$.
		If
		\begin{equation*}
			m\ge \max\Bigl\{804n\log\Bigl(\frac{12n}{\delta}\Bigr),\frac{6}{\delta}\Bigr\}
		\end{equation*}
		nodes $Y=(y_1,\ldots,y_m)$ are drawn independently according to $\varrho_n\mu$, then, with probability at least $1-\delta$, there exist weights $\tfrac{1}{2m\varrho_n(y_i)}\le w_i\le\frac{3}{2m\varrho_n(y_i)}$ with $\sum_{i=1}^m w_i=1$ such that $Q_{\mathbf{w}}(Y)$ is exact on $V_n^\prime$ and
		\begin{equation*}
			%\label{eq:intro_RKHS_error}
			\sup_{\|f\|_{\cH}\le1}
			\biggl|
			\int_{\Omega} f\dint\mu-Q_{\mathbf{w}}(Y)f
			\biggr|
			\le
			(1+\sqrt3)\,\sigma_{n+1}
			+  \biggl(\frac{1}{n}\sum_{j>n}\sigma_j^2\biggr)^{1/2} .
		\end{equation*}
	\end{theorem}

    \noindent
	Theorem~\ref{thm:intro_RKHS} is Corollary~\ref{cor:main2} in Section~\ref{sec:RKHS}.
	Hence, positive rules on $m\asymp n\log(\ee n)$ random nodes achieve the error bound $(1+\sqrt3)\sigma_{n+1}+(1/n\sum_{j>n}\sigma_j^2)^{1/2}$ with high probability and explicit constants. An important example of this setting is high-dimensional integration in Sobolev spaces with mixed smoothness $H^{\alpha}_{\rm mix}$ on the $d$-torus $\mathbb{T}^d$ in connection with hyperbolic cross spaces of trigonometric polynomials, see Corollary~\ref{cor:polynomial_log_decay} and the subsequent discussion.
	
	\paragraph{Optimal Richter-Tchakaloff quadrature widths.}
	Our third result removes the logarithmic oversampling.
	To this end, we reduce the random node set from Theorem~\ref{thm:intro_RKHS} by a constructive compression procedure to $O(n)$ nodes, while keeping exactness on $V_n$ and positivity of the weights (Lemma~\ref{lem:compress} and Proposition~\ref{prop:compress}).
	Our subsampling strategy is in the flavor of \cite{DoKrUl23} for least squares approximation. We do not only keep the exact finite dimensional quadrature but also control the tail bound. We state the result in terms of the Richter-Tchakaloff quadrature widths, introduced recently in \cite{BeyondTchak}. They measure best possible worst-case performance of positive rules on $\cH$ quantified by
	\begin{equation*}
		\kappa_m^+(\cH,\mu)
		\coloneqq
		\inf_{\substack{x_1,\ldots,x_m\in\Omega \\w_1,\ldots,w_m\ge 0}}
		\sup_{\|f\|_{\cH}\le 1}
		\biggl|
		\int_\Omega f \dint\mu - Q_{\mathbf{w}}(X)f
		\biggr|,
	\end{equation*}
	whereas the classical quadrature widths $\kappa_n(\cH,\mu)\le\kappa_n^+(\cH,\mu)$, i.e., the $n$-th minimal worst-case errors of quadrature, see, e.g., \cite[Ch.~10]{NoWo10}, \cite[(8.1.2)]{DuTeUl18}, allow arbitrary weights in $\R$.

	\begin{theorem}
		\label{thm:main1}
		Let $(\Omega,\Sigma,\mu)$ be a positive measure space with $\mu(\Omega)<\infty$, and let $\cH$ be a real-valued separable RKHS on $\Omega$ with finite trace.
		Then, there is an absolute constant $C\ge1$ such that
		\begin{align}
			\label{newbound}
			\kappa^{+}_{m}(\cH,\mu)
			\le  \sqrt{\mu(\Omega)}  \biggl( \frac{1}{n} \sum\limits_{j\ge n}\sigma_j^2 \biggr)^{1/2},\qquad n\in\N,
		\end{align}
		as long as $m\ge C n$.
		As a (loose) choice, one may take $C=39$ for a semi-constructive proof.
	\end{theorem}

    \noindent
	The bound~\eqref{newbound} is optimal over the class of all separable RKHS with finite trace, even among rules with arbitrary real weights~\cite{KrVy23}. In this universal sense, positivity of the weights is no restriction.
    Combining the compression with Richter's theorem instead of Theorem~\ref{thm:intro_RKHS} even gives  $C=\phi+1 = \phi^2 \approx 2.62$, where $\phi = (\sqrt{5}+1)/2$ denotes the Golden Ratio, see  Remark~\ref{rem:nonconstructive}, however non-constructively.

	\paragraph{Polynomial decay and mixed smoothness.} For polynomially decaying singular numbers, Theorem~\ref{thm:main1} gives the following.

	\begin{corollary}
		\label{cor:polynomial_log_decay}
		Let $C$ be the absolute constant from Theorem~\ref{thm:main1}, and assume
		$m\geq Cn$ for $n\in\N$.
		If $\sigma_j \asymp j^{-\alpha}(\log(\ee j))^\beta$ and
		$\alpha>1/2$, $\beta\in\R$, then there exists a constant
		$C_{\alpha,\beta}\geq 1$ such that
		\begin{equation*}
			\kappa_m^+(\cH,\mu)
			\leq
			C_{\alpha,\beta}
			\sqrt{\mu(\Omega)}
			n^{-\alpha}(\log(\ee n))^\beta \quad\text{ for all }n\in\N.
		\end{equation*}
	\end{corollary}
	
	\begin{proof}
		Since $m\geq Cn$, we can apply Theorem~\ref{thm:main1}. A direct calculation shows
		\begin{equation*}
			\frac{1}{\sqrt{\mu(\Omega)}} \kappa^{+}_{m}(\cH,\mu)
			\leq \bigg( \frac{1}{n}\sum\limits_{j\ge n}\sigma_j^2 \bigg)^{1/2} \asymp  \bigg( \frac{1}{n}\sum\limits_{j\ge n}j^{-2\alpha}(\log(\ee j))^{2\beta} \bigg)^{1/2} \asymp   n^{-\alpha}(\log(\ee n))^{\beta}. \qedhere
		\end{equation*}
	\end{proof}

    \noindent
	The most prominent instance of Corollary~\ref{cor:polynomial_log_decay} are the periodic Sobolev spaces with dominating mixed smoothness $\cH\coloneqq H^\alpha_{\rm mix}(\T^d)$, $\alpha>1/2$, on the $d$-torus $\T^d\cong [0,1]^d$, equipped with the Lebesgue measure $\mu\coloneqq \lambda^d$. We refer to \cite[Sect.\ 3]{DuTeUl18} for the definition of the spaces and to \cite[Thm.\ 4.5.1]{DuTeUl18} for the decay of the singular numbers (linear widths, $p=q=2$).  
	In fact, $\sigma_j\asymp j^{-\alpha}(\log(\ee j))^{(d-1)\alpha}$, and Corollary~\ref{cor:polynomial_log_decay} yields positive rules with
	\begin{equation}
		\label{eq:mixed}
		\kappa_m^+(H^\alpha_{\rm mix}(\T^d),\lambda)\lesssim_{\alpha,d} m^{-\alpha}(\log m)^{(d-1)\alpha}.
	\end{equation}
	The random rules of Theorem~\ref{thm:intro_RKHS}, i.e., without compression, give the slightly weaker bound $m^{-\alpha}(\log m)^{d\alpha}$ with high probability.
	If $V_n$ is spanned by the real trigonometric system with frequencies in a hyperbolic cross, see \cite[Sect.~2.3]{DuTeUl18}, then $V_n^\prime=V_n$ is translation invariant, so that $\varrho_n\equiv1$, and the nodes in Theorem~\ref{thm:intro_RKHS} are uniformly distributed points.
	\subsection{Discussion and related work}
	This topic has become a very active area of research. We have tried to give an overview of the most relevant contributions from our point of view.
    
    \paragraph{Richter-Tchakaloff and exact rules from i.i.d.\ nodes.}
		The classical theorem of Richter~\cite[Thm.~4]{Richter57} guarantees the existence of a positive quadrature rule with at most $n$ nodes that is exact on $V$, i.e., $Q_{\mathbf w}(X)f=\int_\Omega f\dint\mu$ for all $f\in V$. It is a general measure-theoretic version of Tchakaloff's celebrated theorem~\cite[Thm.~II]{Tch57} from 1957, concerning polynomial spaces on compact subsets of $\R^d$, and in fact predates the latter. Both theorems were published in the same year, yet Richter's result appeared several months earlier. A predecessor for bounded measurable functions was established even before, in 1952, by Rosenbloom~\cite{Ros52}. The historical development is rather intricate. 
	Indeed, Richter's theorem was reproved several times, e.g., by Rogosinski~\cite[Thm.~1]{Rog58} in 1958 and by Bisgaard~\cite[Thm.~5.1]{BerSas12} in 2012, in an article authored by Berschneider and Sasv{\'a}ri, see also Bayer and Teichmann~\cite{BaTe06}. For a more detailed historical account, we refer to di Dio and Schmüdgen~\cite{DiDioSchmudgen22}.

Very recently, B\v{e}l\'{\i}k, Narayan and Turnage~\cite[Thm.~10]{BeNaTu26} obtained non-negative exact rules from randomized least squares with $M\asymp QJ\log(N/\xi)$ samples, where $Q\ge N$ and $J$ is an implicit, problem dependent threshold.
	Under an $L_r$-integrability condition, $r>2$, this becomes $M\asymp (N\log N)^{r/(r-2)}$~\cite[Lem.~5]{BeNaTu26}, i.e., polynomially larger than $N\log N$. However, the utilized sampling density is more suitable for implementation and practical purposes compared to the Lewis density.  

Migliorati and Nobile \cite{MiNo22} consider randomized cubature based on
weighted least squares with i.i.d.\ nodes from the optimal $L_2$ (Christoffel) density, see also \cite[Sect.\ 7]{KaUlVo21}.
Their rules are exact on $V$ with high probability for $m\gtrsim n\log n$, but positive
weights are only guaranteed under the quadratic oversampling $m/\log m\gtrsim n^2$ in 
\cite[Thm.~6]{MiNo22}. In contrast, Lemma~\ref{lem:box} searches for a correction within an
$\ell_\infty$-box, which, via the $L_1$ geometry of the Lewis density, is feasible
already for $m\asymp n\log n$.
  
	Our Theorem~\ref{thm:intro_RT} requires $m\asymp n\log(n/\delta)$ nodes with absolute constants, for every $n$-dimensional $V\subset L_1(\Omega,\mu)$.
	Hayakawa, Lyons and Oberhauser~\cite{HaLyOb23} bound the number of i.i.d.\ samples whose convex hull contains a given point by $(3n+1)/\alpha$, where $\alpha$ is the Tukey depth of this point.
	For the integral vector of a basis of $V$, the Tukey depth can be as small as $1/n$, e.g., for piecewise constant functions on $n$ cells of equal measure, which only gives $O(n^2)$ samples.
	In the same example, any rule exact on $V$ needs a node in each cell, so that, by the coupon collector problem, $\gtrsim n\log n$ i.i.d.\ nodes are necessary.
	Hence, the oversampling in Theorem~\ref{thm:intro_RT} is optimal.
	Oettershagen~\cite[Ch.~4]{Oe17} constructs exact rules on uniformly distributed random points for trigonometric polynomials on a hyperbolic cross, also with logarithmic oversampling, but with stable minimal-norm weights that are not necessarily positive.
	
A widely used computational approach is Carath\'eodory-Tchakaloff compression by Piazzon, Sommariva, and Vianello
\cite{SommarivaVianello15,PiazzonSommarivaVianello17}. A discrete measure that is
already exact on $V$ is reduced to at most
$\dim V$ nodes with positive weights via linear programming or non-negative least
squares. This presupposes an exact initial rule, which Theorem~\ref{thm:intro_RT} provides
from random nodes. Note that our compression step in Section \ref{sec:compression}  controls  the size of the  weights and the error
outside of $V$, see Lemma~\ref{lem:compress}.

    On the sphere, spherical $t$-designs have been computed by Gräf and Potts~\cite{GrPo11,Gr13} using nonconvex optimization based on fast spherical Fourier transforms, see also~\cite{Wo18}.
	Optimal weights for random points on manifolds are studied in~\cite{EhGrOa19}.
	Positive rules of degree $t$ exist on every point set with covering radius at most $c_d/t$, see~\cite{MhNaWa01}, hence on $m\asymp_d t^d\log t$ random points~\cite{BrReSaSlWaWo18,ZhZh25}. Similar techniques have been used in \cite{FiHiJaUl24}. Theorem~\ref{thm:intro_RT} adds two-sided weight bounds~\eqref{eq:intro_almost_design} with explicit constants and weights computable by linear programming, using only weak geometric  properties of $\mathbb{S}^d$, see Remark \ref{rem:sphere} below. 
    
	\paragraph{Existing general bounds for quadrature widths.}
	A fundamental bound, which does not even require separability, is of Monte Carlo type~\cite[Thm.~4.7]{BeyondTchak}, namely
	\begin{align}
		\label{bound1}
		\kappa_n^+(\cH,\mu)
		\le  \sqrt{\mu(\Omega)}
		\Bigl(\frac{2}{n} \tr(K) \Bigr)^{1/2}.
	\end{align}
	The same order $n^{-1/2}$ is achieved by the positive rules considered in~\cite[Ch.~10]{NoWo10}, namely QMC rules with equal weights, properly normalized QMC rules, and importance sampling with weights proportional to $K(x,x)^{-1/2}$, see~\cite[Thms.~10.4, 10.7, 10.10]{NoWo10}.
	According to~\cite[NR~10.7:1]{NoWo10}, it is the order one should expect as a lower bound in typical applications.
	Our rules start from the  importance sampling weights $w_i^0=1/(m\varrho(y_i))$, but the correction within $[\frac12w_i^0,\frac32w_i^0]$, which makes them exact on $V_n^\prime$ followed by a compression step, lifts the order from $n^{-1/2}$ to the optimal one in~\eqref{newbound}.
	In the separable case, it is further known from~\cite{BeyondTchak} that
	\begin{align*}
		%\label{bound2}
		\kappa_n^+(\cH,\mu)
		\le
		2\sqrt{\mu(\Omega)}\biggl(\sum_{j\ge n}\sigma_j^2\biggr)^{1/2}.
	\end{align*}
	For classes $\mathcal{F}$ of bounded $\mu$-measurable functions, Novak~\cite{No86} showed
	\begin{align}
		\label{bound3}
		\kappa^{+}_{n+1}(\mathcal{F},\mu) \le 2\mu(\Omega) d_n(\mathcal{F})_{\infty},
	\end{align}
	where $d_n(\mathcal{F})_{\infty}$ is the $n$-th Kolmogorov width of $\mathcal{F}$ in the uniform norm, see~\cite[p.~45]{DuTeUl18}.
	None of these bounds captures the asymptotic behavior of $\kappa_n^+(\cH,\mu)$ correctly~\cite[Rem.~4.5]{BeyondTchak}. In particular,~\eqref{bound3} loses a logarithmic factor for mixed smoothness, see~\eqref{eq:mixed_Novak} below and~\cite[Thm.~4.3.8]{DuTeUl18}.
	
	For the classical widths $\kappa_n(\cH,\mu)$ with signed weights, integrating the weighted least squares approximant from~\cite[Thm.~1]{DoKrUl23}, see also~\cite{KrUl21}, yields the sharp universal bound
	\begin{align*}
		%\label{est:classic_kappa}
		\kappa_{n}(\cH,\mu)
		\le
		\sqrt{\mu(\Omega)}\biggl( \frac{1}{cn} \sum_{j\ge cn}\sigma_j^2
		\biggr)^{1/2}
	\end{align*}
	with an absolute constant $0<c\le 1$, which cannot be improved for general $\cH$ with given singular numbers~\cite{KrVy23}.
	Theorem~\ref{thm:main1} shows that $\kappa_n^+(\cH,\mu)$ satisfies the same bound.
	This gives a partial answer to~\cite[Open Problem~45]{NoWo10}, which asks for the RKHS in which positive linear functionals admit optimal algorithms with non-negative weights.
	In fact, for integration on separable RKHS with finite trace, positive weights are no restriction as long as only the singular numbers are known.
	For individual spaces, a gap between $\kappa_n^+(\cH,\mu)$ and $\kappa_n(\cH,\mu)$ is not excluded.
	It can be exponential in the dimension, see the example in~\cite[Sect.~10.6]{NoWo10}.
	
	\paragraph{Polynomial decay and (mixed) Sobolev smoothness.}
	We compare~\eqref{eq:mixed} with the literature on $H^\alpha_{\rm mix}(\T^d)$.
	\begin{itemize}
		\item \emph{Novak's bound.}
		Since $d_m(H^\alpha_{\rm mix}(\T^d))_\infty\lesssim_{\alpha,d} m^{-\alpha}(\log m)^{(d-1)\alpha+1/2}$~\cite[Thm.~4.3.8]{DuTeUl18},~\eqref{bound3} gives
		\begin{equation}
			\label{eq:mixed_Novak}
			\kappa_m^+(H^\alpha_{\rm mix}(\T^d),\lambda)\lesssim_{\alpha,d} m^{-\alpha}(\log m)^{(d-1)\alpha+1/2},
		\end{equation}
		which is worse than~\eqref{eq:mixed} by $\sqrt{\log m}$.
		For $d=2$, this loss is intrinsic to Kolmogorov widths by the small ball inequality, see~\cite[p.\ 57]{DuTeUl18}.
		
		\item \emph{Sparse grids} have signed weights and the sharp order $m^{-\alpha}(\log m)^{(d-1)(\alpha+1/2)}$, see \cite[Thm.~8.3.1]{DuTeUl18},~\cite{DuUl15}.
		
		\item \emph{Lattice rules.} Korobov's cubature formulas, i.e., rank-1 lattice rules with optimal coefficients, give $m^{-\alpha}(\log m)^{\alpha d}$, which Bakhvalov improved to $m^{-\alpha}(\log m)^{\alpha(d-1)}$ for $\alpha>1$, see~\cite[Sect.~8.9]{DuTeUl18} and~\cite{Te93}. Hence, for $d\ge3$, the best known lattice rules attain the same order as~\eqref{eq:mixed}, and optimal rank-1 lattice rules are not known~\cite[Sect.~8.4]{DuTeUl18}.
		Component-by-component constructions, see~\cite{CoKuNu06,NuCo06}, achieve $m^{-\alpha+\varepsilon}$ for every $\varepsilon>0$, see \cite[Thm.~5.12]{DiKuSl13} and \cite{DiKrPi22}.
         In $d=2$, the Fibonacci lattice is optimal~\cite[Thm.~8.4.1]{DuTeUl18}.

		\item \emph{Random points.}
		Uniform random points with optimal weights were observed to achieve the main rate $m^{-\alpha}$ in~\cite[Ch.~4]{Oe17}, where $m^{-\alpha+1/2}(\log m)^{\alpha d-1/2}$ was proved.
		The order $m^{-\alpha}(\log m)^{\alpha d}$ was shown in~\cite[Cor.~8.6]{KaUlVo21} for least squares weights, which may be negative.
		Theorem~\ref{thm:intro_RKHS} gives this order with positive weights summing to one on plain uniform random points, and Theorem~\ref{thm:main1} improves it by $(\log m)^\alpha$.
		
		\item \emph{Digital nets and Frolov's method.}
		Frolov's cubature~\cite{Fr76,UlUl16}, order-2 digital nets for $1/2<\alpha<2$, see \cite{HiMaOeUl16}, and higher order digital nets and sequences~\cite{Di08,DiPi10,GoSuYo17,GoSuYo18} attain the optimal order $m^{-\alpha}(\log m)^{(d-1)/2}$ with positive weights, see~\cite[Thms.~8.2.1, 8.5.1 and Sect.~8.7]{DuTeUl18}.
		These are, to our knowledge, the only constructions in $d\geq 3$ that outperform~\eqref{eq:mixed} in a special case. Note that there are so far no counterparts of these methods for general RKHS.   
		
		\item \emph{Isotropic Sobolev spaces.}
        For $H^s(\Omega)$, $s>d/2$, on bounded domains or compact manifolds $\Omega$ such as $\mathbb{S}^d$, we have $\sigma_j\asymp j^{-s/d}$, and Theorem~\ref{thm:main1} gives the optimal order $m^{-s/d}$ with positive weights, while Theorem~\ref{thm:intro_RKHS} gives $(m/\log m)^{-s/d}$ on random points.
        On compact Riemannian manifolds, this order is known: positive rules with $m\asymp r^d$ nodes that are exact for all Laplace-Beltrami eigenfunctions with eigenvalues less than $r^2$ exist and have worst-case error $\lesssim r^{-s}$~\cite[Lem.~2.11, Thm.~2.12]{BrChCoGiSeTr14}.
        Ehler et al.~\cite[Thms.~3 and 5]{Ehler2021} extend this by reweighting to measures $\rho\,\mathrm{d}x$ with density $\rho\in H^s$, even with weights summing to one, and show that the order $m^{-s/d}$ is sharp.
        Theorem~\ref{thm:main1} recovers the upper bound without using the manifold structure and for every bounded density $\rho$, since then the singular numbers of $H^s\hookrightarrow L_2(\rho\,\mathrm{d}x)$ are still $\lesssim j^{-s/d}$.
        Krieg and Sonnleitner~\cite{KrSo24a,KrSo24b}, see also~\cite{KrNoSo22}, showed that i.i.d.\ uniform points are asymptotically optimal for recovery and, with suitable weights obtained from moving least squares, for integration, closing a logarithmic gap in~\cite{EhGrOa19}.
	\end{itemize}
	The point of Theorems~\ref{thm:intro_RKHS} and~\ref{thm:main1} is that they hold for every separable RKHS with finite trace and use no structure of $\Omega$, $\mu$, or $K$ beyond the singular system.
	
	\paragraph{Further literature.}
	 The Richter-Tchakaloff quadrature widths and the bounds~\eqref{bound1}-\eqref{bound3} were studied in~\cite{BeyondTchak}. 
	In kernel quadrature, Belhadji, Bardenet and Chainais~\cite{BeBaCh19} sample nodes from a projection determinantal point process (DPP) associated with the leading eigenfunctions and bound the error with optimal, possibly negative, weights in terms of the tail sum of the eigenvalues (without the factor $1/n$). 
    
    Bach~\cite{Bach17} samples $m=\mathcal{O}(n\log n)$ i.i.d.\ nodes from a ridge leverage
score density and obtains, with regularized least squares weights, the error
$\mathcal{O}(\sigma_n)$. This works with a regularity assumption on the decay of the singular
numbers. Positivity of the weights is explicitly not covered there. In contrast, our
rules have positive weights, are exact on $V_n'$, require no assumption on the decay of
the singular numbers, and achieve the optimal order with $\mathcal{O}(n)$ nodes after
compression (Theorem~\ref{thm:main1}). 
	
    Positive weights via random subsampling and recombination are obtained by Hayakawa,
Oberhauser and Lyons~\cite{HaObLy22}, with expected squared worst-case error bounded by
the full eigenvalue tail $\sum_{j\ge n}\sigma_j^2$ plus a Monte Carlo term, i.e., without
the factor $1/n$.
   
Glaubitz~\cite{Glaubitz23} shows that weighted least squares cubature on equidistributed
deterministic points becomes positive and exact for sufficiently many points, without an
explicit bound on their number. Numerically, a quadratic number of points in $\dim V$ is
observed to be necessary.

In the very recent contribution Hayakawa \cite{Hayakawa26} the nodes a drawn directly from $\mu$. The error bounds are stated in terms of the tail sum of squared singular values but without the leading factor $1/n$.
    
	\paragraph{Notation.}
	For $a,b\ge0$, we write $a\lesssim b$ if there is a constant $c>0$ with $a\le cb$, where the dependence of $c$ on parameters is indicated by subscripts, and $a\asymp b$ if $a\lesssim b$ and $b\lesssim a$.
	Vectors $\mathbf{a}\in\R^n$ and elements of $\ell_2(\N)$ are treated as column vectors, $\|\cdot\|_2$ denotes the Euclidean (or $\ell_2$-) norm, and $\|\mathbf{A}\|_{2\to2}$ the spectral norm of a matrix (or operator) $\mathbf{A}$. We further write $\mathbf{A}^\top$ and $\mathbf{A}^\ast$ for the transpose and conjugate transpose of $\mathbf{A}$, respectively. 
	For symmetric positive semidefinite matrices (or operators) $\mathbf{A},\mathbf{B}$, we write $\lambda_{\max}(\mathbf{A})$ for the largest eigenvalue of $\mathbf{A}$ and $\mathbf{A}\preceq\mathbf{B}$ if $\mathbf{B}-\mathbf{A}$ is positive semidefinite (Loewner order).
	$\mathbf{I}_n$ denotes the $n\times n$ identity matrix, $\mathbf{1}_\Omega$ the constant one function on $\Omega$, and $\E$ and $\Prob$ expectation and probability. If necessary, a subscript indicates with respect to which random variables they are taken.
	
	\section{A probabilistic Richter-Tchakaloff construction}\label{sec:RandomRules}
	
	Throughout this section, $(\Omega,\Sigma,\mu)$ is an arbitrary measure space and $V\subset L_1(\Omega,\mu)$ is a space of real-valued functions on $\Omega$ with $\dim V=n$.
	The analysis is based on the $L_1$-Lewis density of $V$ given by Proposition~\ref{prop:lewis}, which is a reformulation of~\cite[Thm.~1]{Lewis78} for $p=1$. The $L_1$-Lewis density stems from Banach space geometry. Using a corresponding change of measure/density leads to an embedding of $ V\subset L_1$ into $\ell_1^m$, see \cite{Talagrand90}. For its computation in the discrete case, see \cite{CohenPeng15}.
	
	\begin{proposition}
		\label{prop:lewis}
		Let $V\subset L_1(\Omega,\mu)$ with $\dim V=n\in\N$, and assume that no nonzero $f\in V$ vanishes $\mu$-almost everywhere.
		There exists a basis $h_1,\ldots,h_n$ of $V$ such that
		\begin{equation}
			\label{eq:lewisid}
			\int_{\{\varrho_{\rm L}>0\}}\frac{h_jh_k}{\varrho_{\rm L}}\dint\mu=\delta_{jk},
			\qquad\text{where}\qquad
			\varrho_{\rm L}\coloneqq \biggl( \frac{1}{n}\sum_{j=1}^{n} h_j^2 \biggr)^{1/2} .
		\end{equation}
		In particular, $\varrho_{\rm L}$ is a probability density, the so-called $L_1$-Lewis density of $V$.
	\end{proposition}
	\begin{proof}
		Since no nonzero element of $V$ vanishes $\mu$-almost everywhere, $V$ can be identified with its image in $L_1(\Omega,\mu)$, and every function in~\cite[Thm.~1]{Lewis78} has a unique representative in $V$.
		With this identification, let $f_1,\ldots,f_n\in V$ and $f=(\sum_{j=1}^n f_j^2)^{1/2}$ be the functions from~\cite[Thm.~1]{Lewis78} for $p=1$, and define $h_j\coloneqq \sqrt{n}f_j$, so that $\varrho_{\rm L}=f$.
		Let now $\mathbf{e}_1,\ldots,\mathbf{e}_n$ denote the unit vectors in $\R^n$.
		Inserting first $\mathbf x\coloneqq \mathbf e_j$ in~\cite[Thm.~1]{Lewis78}, and then $\mathbf x\coloneqq \mathbf{e}_j+\mathbf{e}_k$, yields
		\begin{align*}
			\int_{\{\varrho_{\rm L}>0\}}\frac{h_j^2}{\varrho_{\rm L}}\dint\mu = 1 \quad\text{and}\quad
			\int_{\{\varrho_{\rm L}>0\}}\frac{(h_j+h_k)^2}{\varrho_{\rm L}}\dint\mu = 2 .
		\end{align*}
		From this, the relation~\eqref{eq:lewisid} readily follows.
		Summing the relation over $j=k$ finally gives
		\begin{equation*}
			\int_\Omega\varrho_{\rm L}\dint\mu=\frac1n\sum_{j=1}^n\int_{\{\varrho_{\rm L}>0\}}\frac{h_j^2}{\varrho_{\rm L}}\dint\mu=1. \qedhere  
		\end{equation*}
	\end{proof}

    \noindent
	Note that all functions $f\in V$ vanish on $\{\varrho_{\rm L}=0\}$ by definition of $\varrho_{\rm L}$. 
	As a useful technical lemma, let us also record the following norm estimates.
	
	\begin{lemma}
		\label{lem:lewis}
		For $\mathbf{a}\coloneqq (a_1,\ldots,a_n)\in\R^n$, define $f_{\mathbf{a}}\coloneqq \sum_{j=1}^{n} a_jh_j \in V$.
		Then, it holds
		\begin{equation}
			\label{eq:lewisprops}
			\begin{gathered}
				\|f_{\mathbf{a}}\|_{L_1(\mu)}
				\le \|\mathbf{a}\|_2 
				\le \sqrt{n} \|f_{\mathbf{a}}\|_{L_1(\mu)}, \\
				|f_{\mathbf{a}}(x)|
				\le \|\mathbf{a}\|_2\sqrt{n} \varrho_{\rm L}(x)
				\le n \|f_{\mathbf{a}}\|_{L_1(\mu)} \varrho_{\rm L}(x).
			\end{gathered}
		\end{equation}
	\end{lemma}
	\begin{proof}
		Since all $f_{\mathbf{a}}\in V$ vanish on $\{\varrho_{\rm L}=0\}$ and by \eqref{eq:lewisid}, we have
		\begin{equation*}
			\|f_{\mathbf{a}}\|_{L_1(\mu)}=\int_{\{\varrho_{\rm L}>0\}} |f_{\mathbf{a}}|\dint\mu \quad\text{and}\quad
			\|\mathbf{a}\|_2^2=\int_{\{\varrho_{\rm L}>0\}} f_{\mathbf{a}}^2/\varrho_{\rm L} \dint\mu .
		\end{equation*}
		Thus, applying Hölder's inequality with $p=q=2$ to $\|f_{\mathbf{a}}\|_{L_1(\mu)} = \int_{\{\varrho_{\rm L}>0\}}(f_{\mathbf{a}}/\sqrt{\varrho_{\rm L}})\sqrt{\varrho_{\rm L}}\dint \mu$ gives $\|f_{\mathbf{a}}\|_{L_1(\mu)}\le\|\mathbf{a}\|_2$ since $\varrho_L$ is a probability density. 
		Applying the Hölder inequality to $\sum_{j=1}^{n} a_jh_j(x)$ yields $|f_{\mathbf{a}}(x)|\le\|\mathbf{a}\|_2(\sum_{j=1}^{n}h_j(x)^2)^{1/2}=\|\mathbf{a}\|_2\sqrt n\varrho_{\rm L}(x)$. 
		Consequently, we get $\|\mathbf{a}\|_2^2=\int_{\{\varrho_{\rm L}>0\}} f_{\mathbf{a}}^2/\varrho_{\rm L} \dint\mu\le\|\mathbf{a}\|_2\sqrt{n}\|f_{\mathbf{a}}\|_{L_1(\mu)}$, namely $\|\mathbf{a}\|_2\le\sqrt{n}\|f_{\mathbf{a}}\|_{L_1(\mu)}$.
		This also yields the last missing inequality in \eqref{eq:lewisprops} by insertion.
	\end{proof}
	
	\begin{remark}
    \label{rem:2.3}
		The $L_1$-Lewis density is the $L_1$-analogue of the Christoffel density $\varrho_C$, sometimes called $L_2$-Lewis density, associated to finite-dimensional subspaces $V\subset L_2(\Omega,\mu)$. For an orthonormal basis $(b_1,\ldots,b_n)$ of an $n$-dimensional real-valued space $V$, it is given by
		\begin{equation*}
			\varrho_C\coloneqq\frac{1}{n}\sum_{j=1}^{n} b_j^2
		\end{equation*}
		and usually differs from $\varrho_L$.
		For instance, if $V=\spann\{f\}$, Proposition~\ref{prop:lewis} implies that $\varrho_{\rm L}=|f|/\|f\|_{L_1(\mu)}$, whereas $\varrho_C=f^2/\|f\|_{L_2(\mu)}^2$.
	\end{remark}

    \noindent
	Let $h_1,\ldots,h_n$ and $\varrho_{\rm L}$ denote the Lewis basis and the Lewis density of $V$ from Proposition~\ref{prop:lewis}.
	Further, let $\varrho$ be any probability density with respect to $\mu$ such that
	\begin{equation}
		\label{eq:rho_admissible}
		\varrho_{\rm L}\le \gamma \varrho
	\end{equation}
	on $\Omega$ for some constant $\gamma\ge1$.
	Admissible choices are, e.g., $\varrho=\varrho_{\rm L}$ itself with $\gamma=1$, or mixtures $\varrho=\gamma^{-1}\varrho_{\rm L}+(1-\gamma^{-1})\tilde\varrho$ with an arbitrary probability density $\tilde\varrho$, such as the mixture density~\eqref{mix_density} in Section~\ref{sec:RKHS}.
	The goal of this section is to analyze quadrature rules $Q_{\mathbf{w}}(Y)$ on $V$ of the form
	\begin{equation*}
		Q_{\mathbf{w}}(Y) f\coloneqq  \sum_{i=1}^m w_i f(y_i), 
	\end{equation*}
	see~\eqref{eqdef:Q_{mathbf{w}}(X)}, with non-negative weights $\mathbf{w}=(w_1,\ldots,w_m)\ge0$, where the nodes $Y\coloneqq (y_1,\ldots,y_m)$ are drawn independently according to $\varrho\mu$.
	Note that, by~\eqref{eq:rho_admissible}, all $f\in V$ vanish on $\{\varrho=0\}\subset\{\varrho_{\rm L}=0\}$, and we use the convention that quotients by $\varrho$ are set to $0$ on $\{\varrho=0\}$.
	The main result is the subsequent theorem. It should be compared to~\cite[Thm.~10]{BeNaTu26}, which gives another probabilistic construction of Richter-Tchakaloff rules for $V$, however in our case a fixed log-oversampling of the nodes $Y$ with respect to the dimension of $V$ suffices.
	
	\begin{theorem}
		\label{thm:Prob_Tchak}
		Let $V\subset L_1(\Omega,\mu)$ be a space of real-valued functions on $\Omega$ with $\dim V=n$, and let $\varrho$ be a probability density satisfying~\eqref{eq:rho_admissible} for some $\gamma\ge1$.
		If $m\ge 67\,\gamma n$ nodes $Y=(y_1,\ldots,y_m)$ are drawn independently according to $\varrho\mu$, then, with probability at least
		\begin{equation*}
			%\label{eq:Prob_Tchak_prob}
			1-4n\exp\Bigl(-\frac{m}{134\,\gamma n}\Bigr),
		\end{equation*}
		there exist weights $\mathbf{w}=(w_1,\ldots,w_m)$ with
		\begin{equation*}
			\frac{1}{2m\varrho(y_i)}\le w_i\le\frac{3}{2m\varrho(y_i)},\qquad i=1,\ldots,m,
		\end{equation*}
		such that $Q_{\mathbf{w}}(Y)$ is a Richter-Tchakaloff rule for $V$, i.e., $Q_{\mathbf{w}}(Y)f=\int_\Omega f\dint\mu$ for all $f\in V$.
		If, in addition, $\mathbf{1}_{\Omega}\in V$ and $\mu(\Omega)=1$, then $\sum_{i=1}^m w_i=1$, where $\mathbf{1}_{\Omega}$ denotes the $1$-function on $\Omega$.
	\end{theorem}
	
	\begin{corollary}
		\label{cor:Prob_Tchak}
		In the setting of Theorem~\ref{thm:Prob_Tchak}, the quadrature $Q_{\mathbf{w}}(Y)$ is a Richter-Tchakaloff rule for $V$ with probability at least $1-\delta$, with a free parameter $\delta\in(0,1)$, if 
		\begin{equation*}
			m\ge 134 \gamma n \log\bigl(\frac{4n}{\delta}\bigr).
		\end{equation*}
		In particular, if sampling from the Lewis density, where $\gamma=1$, drawing $m\ge 134 n\log(40n)$ nodes yields a Richter-Tchakaloff rule with probability at least $90\%$.
	\end{corollary}
	
	\begin{remark}[Almost spherical designs]
		\label{rem:sphere}
		Let $\mu$ be the normalized surface measure on $\mathbb{S}^d\subset\R^{d+1}$ and $V=\Pi_t$ the space of spherical polynomials of degree at most $t$, so that $n=\dim\Pi_t\asymp_d t^d$.
		Since $V$ and $\mu$ are invariant under rotations and the $L_1$-Lewis density is unique up to $\mu$-equivalence, see ~\cite[Thm.\ 1]{Lewis78}, $\varrho_{\rm L}$ is rotation invariant and hence $\varrho_{\rm L}\equiv 1$ (up to $\mu$-equivalence).
		Hence, Corollary~\ref{cor:Prob_Tchak} applies with $\varrho\equiv1$ and $\gamma=1$: if $m\ge134\,n\log(4n/\delta)\asymp_d t^d\log(t/\delta)$ nodes are drawn independently and uniformly on $\mathbb{S}^d$, then, with probability at least $1-\delta$, there are weights
		\begin{equation*}
			\frac{1}{2m}\le w_i\le\frac{3}{2m},\qquad \sum_{i=1}^m w_i=1,
		\end{equation*}
		such that $Q_{\mathbf w}(Y)$ is exact on $\Pi_t$.
		By Corollary~\ref{cor:near_equal}, $m\ge 34\tau^{-2}n\log(4n/\delta)$ nodes even admit weights in $[\frac{1-\tau}{m},\frac{1+\tau}{m}]$ for $\tau\in(0,\frac12]$, i.e., the random points form an almost spherical $t$-design.
	\end{remark}
	
	\noindent
	The proof of Theorem~\ref{thm:Prob_Tchak} consists of three steps. 
	First, we reduce the setting to the case where no nonzero element of $V$ vanishes $\mu$-almost everywhere (Remark~\ref{rem:reduction}). This is necessary to make Proposition~\ref{prop:lewis} applicable, namely in case
	\begin{equation}
		\label{eqdef:Nullspace}
		N\coloneqq \{f\in V: f=0\ \mu\text{-a.e.}\}\neq\{0\}.
	\end{equation}
	Second, we show that the rule 
	\begin{equation}
		\label{eqdef:ini_rule}
		Q^0(Y) f\coloneqq  \sum_{i=1}^m w_i^0f(y_i), \qquad w_i^0 \coloneqq  \frac{1}{m\varrho(y_i)},
	\end{equation} 
	approximates both $\int_\Omega f\dint\mu$ and $\int_\Omega|f|\dint\mu$ uniformly over $f\in V$ with $\|f\|_{L_1(\mu)}\le1$, up to errors of $\frac14$ and $\frac12$, respectively (Proposition~\ref{prop:conc}). 
	This rests on an $L_p$-bound for the support function of the $L_1$-coefficient body (Lemma~\ref{lem:trunc}) and on moment estimates for the suprema of the associated empirical processes (Lemma~\ref{lem:moments}).
	Third, we adapt the weights within the box $[\frac12 w_i^0,\frac32 w_i^0]$ so that the rule becomes exact on $V$ (Lemma~\ref{lem:box}). Note that, since $\varrho(y_i)>0$ almost surely, the weights $w_i^0$ in~\eqref{eqdef:ini_rule} are almost surely well-defined.

\begin{remark}[Reduction in case $N\neq\{0\}$]\label{rem:reduction}
If $N\neq\{0\}$, see~\eqref{eqdef:Nullspace}, choose a complement $W$ with $V=W\oplus N$,
so that $\dim W<n$ and no nonzero element of $W$ vanishes $\mu$-almost everywhere.
Then the Lewis basis and density of $V$ are understood as those of $W$, with
$\varrho_L\coloneqq0$ if $W=\{0\}$. This is justified as follows. The set
$\Omega_N\coloneqq\bigcup_{k\le\ell}\{b_k\neq0\}$, where $b_1,\dots,b_\ell$ is a basis
of $N$, is a $\mu$-null set, so $Y\cap\Omega_N=\emptyset$ almost surely. On this event,
$Q_{\mathbf w}(Y)f=0=\int_\Omega f\dint\mu$ for all $f\in N$ and all weights $\mathbf w$,
so every rule that is exact on $W$ is exact on $V$. Since the condition on $m$ and the
failure probability only improve when $n$ decreases, it suffices to treat $W$; for
$W=\{0\}$, any rule $Q^0(Y)$ is almost surely exact on $V$. Hence, we assume from now
on that no nonzero element of $V$ vanishes $\mu$-almost everywhere.
\end{remark}
    
	\subsection{An \texorpdfstring{$L_p$}{Lp}-bound for the \texorpdfstring{$L_1$}{L1}-coefficient body}
	For the proof of Theorem \ref{thm:Prob_Tchak}, we define $B_1(V)\coloneqq \{f_{\mathbf{a}}:\mathbf{a}\in B\}$ based on the $L_1$-coefficient body
	\begin{align}
		\label{eqdef:coeff_body}
		B\coloneqq \bigl\{\mathbf{a}\in\R^n:\|f_{\mathbf{a}}\|_{L_1(\mu)}\le1\bigr\},
	\end{align}
    where $f_{\mathbf{a}}$ is the linear combination  from Lemma \ref{lem:lewis}. The key tool for analyzing the rules $Q^0(Y)$ is a deterministic estimate for the support function
	\begin{equation}
		\label{eqdef:s_B(mathbf{z})}
		s_B(\mathbf{z})\coloneqq \sup_{\mathbf{a}\in B}\langle \mathbf{a},\mathbf{z}\rangle=\sup_{\mathbf{a}\in B}|\langle \mathbf{a},\mathbf{z}\rangle| ,\quad \mathbf{z}\in\R^n,
	\end{equation}
	where the second equality is due to the symmetry of $B$.
	For every $x\in\Omega$ with $\varrho_{\rm L}(x)>0$, let
	\begin{equation}
		\label{eqdef:mathbf{u}(x)}
		\mathbf{u}(x)\coloneqq \big(h_j(x)/\varrho_{\rm L}(x)\big)^{\top}_{j\le n}\in\R^n,
	\end{equation}
	so that $\|\mathbf{u}(x)\|_2^2=n$ by~\eqref{eq:lewisid}.
	For $\mathbf{z}\in\R^n$, we then define the function
	\begin{equation}
		\label{eqdef:G_{mathbf{z}}}
		G_{\mathbf{z}}(x)\coloneqq \begin{cases}
			\langle \mathbf{z},\mathbf{u}(x)\rangle &\text{if }\varrho_{\rm L}(x)>0,\\
			0 &\text{if }\varrho_{\rm L}(x)=0.
		\end{cases}
	\end{equation}
	
	\begin{lemma}
		\label{lem:trunc}
		For every $\mathbf{z}\in\R^n$ and every $p\in[1,\infty)$, it holds that
		\begin{equation}\label{eq:trunc}
			s_B(\mathbf{z})^p
			\le
			n\int_\Omega |G_{\mathbf{z}}|^p \varrho_{\rm L}\dint\mu.
		\end{equation}
	\end{lemma}
	\begin{proof}
		By~\eqref{eq:lewisid}, it holds for every $\mathbf{a}\in\R^n$ that
		\begin{equation*}
        %\label{eq:SBz}
			\int_\Omega f_{\mathbf{a}} G_{\mathbf{z}}\dint\mu=\sum_{j,k=1}^n a_jz_k\int_{\{\varrho_{\rm L}>0\}}\frac{h_jh_k}{\varrho_{\rm L}}\dint\mu
			=\langle \mathbf{a},\mathbf{z}\rangle.
		\end{equation*}
		Set $\varphi_{\mathbf a}\coloneqq f_{\mathbf a}/\varrho_{\rm L}$ on $\{\varrho_{\rm L}>0\}$ and $\varphi_{\mathbf a}\coloneqq0$ otherwise, so that $\langle\mathbf a,\mathbf z\rangle=\int_\Omega\varphi_{\mathbf a}G_{\mathbf z} \varrho_{\rm L}\dint\mu$.
		For $\mathbf a\in B$, \eqref{eq:lewisprops} gives $|\varphi_{\mathbf a}|\le n\|f_{\mathbf a}\|_{L_1(\mu)}\le n$.
		For $p=1$, the claim follows directly from $|\varphi_{\mathbf a}|\le n$.
		For $p>1$, let $q\coloneqq p/(p-1)$.
		Then, it holds
		\begin{equation*}
			\int_\Omega|\varphi_{\mathbf a}|^q\varrho_{\rm L}\dint\mu\le n^{q-1}\int_\Omega|\varphi_{\mathbf a}|\varrho_{\rm L}\dint\mu = n^{q-1} \|f_{\mathbf a}\|_{L_1(\mu)}\le n^{q-1},
		\end{equation*}
		which implies $\|\varphi_{\mathbf a}\|_{L_q(\varrho_{\rm L}\mu)}\le n^{1-1/q}=n^{1/p}$.
		Now, Hölder's inequality with respect to the probability measure $\varrho_{\rm L}\mu$ yields
		\begin{equation*}
			|\langle\mathbf a,\mathbf z\rangle| = \biggl|\int_\Omega\varphi_{\mathbf a}G_{\mathbf z} \varrho_{\rm L}\dint\mu \biggr|\le n^{1/p} \|G_{\mathbf z}\|_{L_p(\varrho_{\rm L}\mu)}
		\end{equation*}
		for all $\mathbf a\in B$.
		Taking the $p$-th power and then the supremum over $\mathbf a\in B$ gives~\eqref{eq:trunc}.
	\end{proof}
	
	\subsection{Quadrature error and simultaneous \texorpdfstring{$L_1$}{L1}-norm discretization}
    %\label{ssec:RelConc}
    
	For fixed $m\in\N$ and $f \in B_1(V)$, we analyze the two centered stochastic processes 
	\begin{align*}
		E_{\rm sign}[f]\coloneqq \sum_{i=1}^m w_i^0f(y_i)-\int_\Omega f\dint\mu,\qquad 
		E_{\rm abs}[f] \coloneqq  \sum_{i=1}^m w_i^0|f(y_i)|-\int_\Omega |f|\dint\mu,
	\end{align*}
	where the $y_i$, $i=1,\ldots,m$, are drawn independently according to $\varrho \mu$.
	Those can be reformulated in terms of $f_\mathbf{a} =\sum_{j=1}^{n} a_jh_j$ and the coefficients $\mathbf{a}\in B$, see~\eqref{eqdef:coeff_body}, as
	\begin{align}
		E_{\rm sign}[\mathbf a]
		&= \sum_{i=1}^m \frac{f_{\mathbf{a}}(y_i)}{m\varrho(y_i)}-\int_\Omega f_{\mathbf{a}}\dint\mu =  \sum_{i=1}^m \frac{1}{m} \biggl(\frac{f_{\mathbf{a}}(y_i)}{\varrho(y_i)}-\int_\Omega f_{\mathbf{a}}\dint\mu\biggr), \label{eq:ProccesSign}\\
		E_{\rm abs}[\mathbf a]
		&= \sum_{i=1}^m \frac{|f_{\mathbf{a}}(y_i)|}{m\varrho(y_i)}-\int_\Omega |f_{\mathbf{a}}|\dint\mu =  \sum_{i=1}^m \frac{1}{m} \biggl(\frac{|f_{\mathbf{a}}(y_i)|}{\varrho(y_i)}-\int_\Omega |f_{\mathbf{a}}|\dint\mu\biggr).\label{eq:ProccesAbs}
	\end{align}

    \noindent
	Our aim is to probabilistically control the suprema of these processes, namely 
	\begin{equation*}
    %\label{eqdef:Z}
		Z_{\rm sign}\coloneqq \sup_{f\in B_1(V)}\biggl|\sum_{i=1}^mw_i^0f(y_i)-\int_\Omega f\dint\mu\biggr|,
		\quad
		Z_{\rm abs}\coloneqq \sup_{f\in B_1(V)}\biggl|\sum_{i=1}^mw_i^0|f(y_i)|-\int_\Omega |f|\dint\mu\biggr|.
	\end{equation*}
	To handle $Z_{\rm sign}$ and $Z_{\rm abs}$ simultaneously, we henceforth use $Z_{\bullet}$ to refer to both variables, $Z_{\rm sign}$ and $Z_{\rm abs}$, at the same time.
	Using the equivalent formulations \eqref{eq:ProccesSign} and \eqref{eq:ProccesAbs} and defining $\tilde{f}_{\mathbf{a}}\coloneqq f_{\mathbf{a}}/\varrho$, we can write 
	\begin{equation}\label{eqdef:Z_reform}
		Z_{\rm sign}=\sup_{\mathbf{a}\in B}\bigg|\frac1m\sum_{i=1}^m \tilde{f}_{\mathbf{a}}(y_i)-\int_\Omega \tilde{f}_{\mathbf{a}}\varrho\dint\mu\bigg|, \qquad
		Z_{\rm abs}=\sup_{\mathbf{a}\in B}\biggl|\frac1m\sum_{i=1}^m |\tilde{f}_{\mathbf{a}}(y_i)| -\int_\Omega |\tilde{f}_{\mathbf{a}}| \varrho\dint\mu\biggr|.
	\end{equation}
	Since $\mathbf a\mapsto\tilde f_{\mathbf a}(y)$ and $\mathbf a\mapsto\int_\Omega f_{\mathbf a}\dint\mu$ are linear and $\mathbf a\mapsto\|f_{\mathbf a}\|_{L_1(\mu)}$ is $1$-Lipschitz by~\eqref{eq:lewisprops}, all processes indexed by $\mathbf a\in B$ in this subsection are continuous in $\mathbf a$.
	Hence, their suprema over $B$ coincide with those over the countable dense set $B\cap\Q^n$, so that the $Z_\bullet$ are measurable.
	Finally, for each $y\in\Omega$, define the column vectors 
	\begin{align}\label{eq:DefH}
		\mathbf{h}(y)\coloneqq (h_j(y))^{\top}_{j\le n}\in\R^n,\qquad  \tilde{\mathbf{h}}(y)\coloneqq \mathbf{h}(y)/\varrho(y)\in\R^n .
	\end{align} 
	Based on \eqref{eq:lewisid} and \eqref{eq:rho_admissible}, we obtain
	\begin{equation}\label{eq:featbounds}
		\|\tilde{\mathbf{h}}(y)\|_2=\frac{\sqrt{n}\varrho_{\rm L}(y)}{\varrho(y)}\le \gamma\sqrt{n},
		\qquad
		\E_{\mathbf{y}}\big[\tilde{\mathbf{h}}(y)\tilde{\mathbf{h}}^{\top}(y)\big]=\int_\Omega\frac{\mathbf{h}\mathbf{h}^{\top}}{\varrho}\dint\mu\preceq\gamma \mathbf{I}_n.
    \end{equation}
	Lemma~\ref{lem:moments} below follows from the Rademacher-process approach of
\cite[Sect.~15.5]{LedouxTalagrand} and \cite{Talagrand90}. Symmetrization and the
contraction principle, see \cite[Lem.~6.3, Thm.~4.12]{LedouxTalagrand} and the appendix here, reduce the problem to a Rademacher average, cf.\ \cite[Prop.~15.16]{LedouxTalagrand}. We bound all moments of this average via the Lewis
basis and a Bernstein-type moment inequality, see Lemma~\ref{lem:bernmom}, which yields the
i.i.d.\ sampling statement with arbitrary failure probability via the ``moments to tails'' Proposition \ref{propA:MomTail}.
	\begin{lemma}[Moments]
		\label{lem:moments}
		Let $m\in\N$.
		For every $p\in[1,m/(64\gamma n)]$, it holds that
		\begin{equation}\label{eq:moments}
			\bigl(\E_{\mathbf y}Z_\bullet^p\bigr)^{1/p}\le (2n)^{1/p}c_\bullet {\ee}^{-11/23}\sqrt{\frac{p\gamma n}{m}},
			\qquad
			\bullet\in\{\mathrm{sign},\mathrm{abs}\},
		\end{equation}
		where $c_{\rm sign}\coloneqq2$ and $c_{\rm abs}\coloneqq4$.
	\end{lemma}

    \noindent
	The factor $(2n)^{1/p}$ is the price for the uniformity over $B$, and it is bounded by an absolute constant as soon as $p\gtrsim\log n$.
	Combined with Markov's inequality in the form of Proposition~\ref{propA:MomTail}, the $p$-th moments also provide the dependence on the failure probability, see the proof of Proposition~\ref{prop:conc} below.
	Bounding all moments of a supremum, absorbing a factor $n^{1/p}$ by the choice $p\gtrsim\log n$, and passing to tail bounds is a classical technique, see e.g.~\cite{Ru99}, \cite[Prop.~7.11]{FoRa13}, as well as~\cite[Cor.~3.2, Prop.~3.6, and Prop.~3.8]{MoUl20}.
	
	\begin{proof}
		\emph{Step 1: We derive the symmetrization bound~\eqref{eq:symm}.}
		For $\mathbf a\in B$, let $g_{\mathbf a}\coloneqq\tilde f_{\mathbf a}$ if $\bullet=\mathrm{sign}$, and $g_{\mathbf a}\coloneqq|\tilde f_{\mathbf a}|$ if $\bullet=\mathrm{abs}$.
		By~\eqref{eqdef:Z_reform}, we have $Z_\bullet=\sup_{\mathbf a\in B}|\frac1m\sum_{i=1}^m g_{\mathbf a}(y_i)-\int_\Omega g_{\mathbf a}\varrho\dint\mu|$, where the supremum may be taken over the countable set $B\cap\Q^n$.
		For a finite set $D\subset B\cap\Q^n$, we apply the symmetrization Theorem~\ref{thmA:Symm} to the independent random vectors $(g_{\mathbf a}(y_i)/m)_{\mathbf a\in D}\in\R^{D}$, $i=1,\ldots,m$, where $\R^D$ is equipped with the maximum norm.
		Letting $D\uparrow B\cap\Q^n$ and using monotone convergence, we obtain
		\begin{equation}\label{eq:symm}
			\E_{\mathbf y}Z_\bullet^p
			\le2^p \E_{\mathbf y,\boldsymbol\epsilon}\sup_{\mathbf a\in B}\biggl|\frac1m\sum_{i=1}^m\epsilon_i g_{\mathbf a}(y_i)\biggr|^p ,
		\end{equation}
		where $\boldsymbol\epsilon=(\epsilon_1,\ldots,\epsilon_m)$ is a vector of independent Rademacher variables, independent of $\mathbf y$, see also~\cite[Lem.~6.3]{LedouxTalagrand} and~\cite[Prop.~3.7]{MoUl20}.
		
		\noindent\emph{Step 2: We reduce the estimate in~\eqref{eq:symm} to the support function in~\eqref{eq:reduction}.}
		Let us now insert $\mathbf z\coloneqq\frac1m\sum_{i=1}^m\epsilon_i\tilde{\mathbf h}(y_i)$ into~\eqref{eqdef:s_B(mathbf{z})}, where $\tilde{\mathbf h}(y_i)$ is defined according to \eqref{eq:DefH}.
		Since $\tilde f_{\mathbf a}(y_i)=\langle\mathbf a,\tilde{\mathbf h}(y_i)\rangle$, the right-hand side of~\eqref{eq:symm} then equals $2^p \E_{\mathbf y,\boldsymbol\epsilon} s_B(\mathbf z)^p$ for $\bullet=\mathrm{sign}$.
		We next consider the case $\bullet=\mathrm{abs}$ and, to this end, introduce the set $T\coloneqq\{(\tilde f_{\mathbf a}(y_i))_{i\le m}:\mathbf a\in B\}\subset\R^m$ for fixed $\mathbf y$. It is a bounded set, since~\eqref{eq:lewisprops} and~\eqref{eq:featbounds} imply
		\begin{align*}
			%\label{eq:Rad_bound}
			\sup_{y\in\Omega}\sup_{\mathbf{a}\in B}|\tilde{f}_{\mathbf{a}}(y)| =\sup_{y\in\Omega}\sup_{\mathbf{a}\in B}|\langle \mathbf{a},\tilde{\mathbf{h}}(y)\rangle|\le \sup_{y\in\Omega}\sup_{\mathbf{a}\in B}\|\mathbf{a}\|_2 \|\tilde{\mathbf{h}}(y)\|_2\le\gamma n .
		\end{align*}
		The function $\varphi\colon t\mapsto|t|$ is $1$-Lipschitz on $\R$ with $\varphi(0)=0$.
		Hence, applying the contraction Theorem~\ref{thmA:Contraction} with $F(s)\coloneqq(4s/m)^p$ and $\varphi_i\coloneqq\varphi$ gives for fixed $\mathbf y$ that
		\begin{align*}
			2^p \E_{\boldsymbol\epsilon}\sup_{\mathbf a\in B}\biggl|\frac1m\sum_{i=1}^m\epsilon_i\varphi\bigl(\tilde f_{\mathbf a}(y_i)\bigr)\biggr|^p
			&=\E_{\boldsymbol\epsilon}F\biggl(\frac12\sup_{t\in T}\biggl|\sum_{i=1}^m\epsilon_i\varphi(t_i)\biggr|\biggr)
			\le\E_{\boldsymbol\epsilon}F\biggl(\sup_{t\in T}\biggl|\sum_{i=1}^m\epsilon_it_i\biggr|\biggr)\notag\\
			&=4^p \E_{\boldsymbol\epsilon}\sup_{\mathbf a\in B}\biggl|\frac1m\sum_{i=1}^m\epsilon_i\tilde f_{\mathbf a}(y_i)\biggr|^p .
		\end{align*}
		Taking the expectation with respect to $\mathbf y$, the right-hand side of~\eqref{eq:symm} is therefore at most $4^p \E_{\mathbf y,\boldsymbol\epsilon} s_B(\mathbf z)^p$ for $\bullet=\mathrm{abs}$.
		We obtain in both cases $\bullet\in\{\mathrm{sign},\mathrm{abs}\}$
		\begin{equation}\label{eq:reduction}
			\E_{\mathbf y}Z_\bullet^p\le c_\bullet^p \,\E_{\mathbf y,\boldsymbol\epsilon} s_B(\mathbf z)^p .
		\end{equation}
		
		\noindent\emph{Step 3: We apply Lemma~\ref{lem:trunc} and Lemma~\ref{lem:bernmom} to obtain \eqref{eq:moments}.}
		Lemma~\ref{lem:trunc} and Tonelli's theorem, applicable since the integrand is non-negative and jointly measurable in $(x,\mathbf y,\boldsymbol\epsilon)$, give
		\begin{equation*}
			\E_{\mathbf y,\boldsymbol\epsilon} s_B(\mathbf z)^p\le n\int_\Omega\varrho_{\rm L}(x) \E_{\mathbf y,\boldsymbol\epsilon}|G_{\mathbf z}(x)|^p\dint\mu(x)
			\le n\sup_{x: \varrho_{\rm L}(x)>0}\E_{\mathbf y,\boldsymbol\epsilon}|G_{\mathbf z}(x)|^p ,
		\end{equation*}
		where the last step uses $\int_\Omega\varrho_{\rm L}\dint\mu=1$.
		For fixed $x$ with $\varrho_{\rm L}(x)>0$,  see~\eqref{eqdef:mathbf{u}(x)} and \eqref{eqdef:G_{mathbf{z}}}, we have
		\begin{equation*}
			G_{\mathbf z}(x)= \langle \mathbf{z},\mathbf{u}(x)\rangle = \sum_{i=1}^mX_i
		\end{equation*}
		with the i.i.d.\ variables $X_i\coloneqq\frac1m \epsilon_i\langle\tilde{\mathbf h}(y_i),\mathbf u(x)\rangle$, $i=1,\ldots,m$, which are centered since $\epsilon_i$ is independent of $y_i$.
		By $\|\mathbf u(x)\|_2=\sqrt{n}$, $|\epsilon_i|=1$, and $\|\tilde{\mathbf h}(y_i)\|_2\le\gamma\sqrt{n}$, see~\eqref{eq:featbounds}, these random variables satisfy $|X_i|\le\gamma n/m$.
		Furthermore, $\sum_{i=1}^m\E_{\mathbf y,\boldsymbol\epsilon} X_i^2\le \gamma n/m$ due to, see again~\eqref{eq:featbounds}, 
		\begin{equation*}
			\E_{\mathbf y,\boldsymbol\epsilon} X_i^2 =  \frac{1}{m^2} \E_{\mathbf{y}}\big[ \mathbf{u}(x)^\top \tilde{\mathbf{h}}(y)\tilde{\mathbf{h}}^{\top}(y)  \mathbf{u}(x)\big] =\frac{1}{m^2} \mathbf{u}(x)^\top \E_{\mathbf{y}}\big[\tilde{\mathbf{h}}(y)\tilde{\mathbf{h}}^{\top}(y)\big]\mathbf{u}(x)\le \frac{\gamma}{m^2}\|\mathbf u(x)\|_2^2 = \frac{\gamma n}{m^2}.
		\end{equation*}
		Lemma~\ref{lem:bernmom} with $v\coloneqq L\coloneqq\gamma n/m$, applicable since $L\sqrt{p/v}=\sqrt{p\gamma n/m}\le\frac18$, thus yields
		\begin{equation*}
			\E_{\mathbf y,\boldsymbol\epsilon}|G_{\mathbf z}(x)|^p\le2\Bigl({\ee}^{-11/23}\sqrt{\frac{p\gamma n}{m}}\Bigr)^p .
		\end{equation*}
		Inserting this into~\eqref{eq:reduction} gives $\E_{\mathbf y}Z_\bullet^p\le 2n\,c_\bullet^p ({\ee}^{-11/23}\sqrt{p\gamma n/m})^p$, which is~\eqref{eq:moments}.
	\end{proof}

    \noindent
	In combination with Proposition~\ref{propA:MomTail}, the previous lemma leads to the main result of this subsection, which is an essential ingredient in the proof of Theorem~\ref{thm:Prob_Tchak}.
	
	\begin{proposition}
		\label{prop:conc}
		If $m\ge 67\,\gamma n$, then, with probability at least $1-4n\exp(-m/(134\gamma n))$, both
		\begin{equation}\label{eq:conc}
			\sup_{f\in B_1(V)}\biggl|\sum_{i=1}^mw_i^0f(y_i)-\int_{\Omega} f\dint\mu\biggr|\le\frac14 \quad\text{and}\quad \sup_{f\in B_1(V)}\biggl|\sum_{i=1}^mw_i^0|f(y_i)|-\int_{\Omega} |f|\dint\mu\biggr|\le\frac12 
		\end{equation}
		hold true.
	\end{proposition}
	\begin{proof}
		Let $\bullet\in\{\mathrm{sign},\mathrm{abs}\}$ and $\alpha_\bullet\coloneqq c_\bullet{\ee}^{-11/23}\sqrt{\gamma n/m}$.
		By Lemma~\ref{lem:moments}, $Z_\bullet$ satisfies the assumption of Proposition~\ref{propA:MomTail} with $\beta\coloneqq2n$, $\nu\coloneqq2$, $p_0\coloneqq1$, and $p_1\coloneqq m/(64\gamma n)>1$, which gives
		\begin{equation}
			\label{eq:fail_prob_2_40}
			\Prob\bigl(Z_\bullet\ge {\ee}^{1/2}\alpha_\bullet u\bigr)\le 2n\,{\ee}^{-u^2/2},\qquad u\in\bigl[1,\sqrt{m/(64\gamma n)}\bigr].
		\end{equation}
		We choose $u\coloneqq\sqrt{m/(64{\ee}^{1/23}\gamma n)}$, which lies in this range, since $m\ge67\gamma n\ge64{\ee}^{1/23}\gamma n$.
		Then ${\ee}^{1/2}\alpha_\bullet u=c_\bullet/8$, which is $\frac14$ for $\bullet=\mathrm{sign}$ and $\frac12$ for $\bullet=\mathrm{abs}$.
		By a union bound, \eqref{eq:conc} fails with probability at most $4n{\ee}^{-u^2/2}=4n\exp(-m/(128{\ee}^{1/23}\gamma n))\le 4n\exp(-m/(134\gamma n))$.
	\end{proof}

\begin{remark}($L_1$-norm discretization with $\log$-oversampling)\label{rem:discretization}
As a by-product, the bound on $Z_{\mathrm{abs}}$ in Proposition \ref{prop:conc} yields a weighted
$L_1$ sampling discretization (Marcinkiewicz-Zygmund) inequality for i.i.d.\ random
points.
For the event of Proposition~\ref{prop:conc}, it holds
\begin{equation*}
    \tfrac12 \|f\|_{L_1(\mu)} \le \sum_{i=1}^m w_i^0 |f(y_i)| \le \tfrac32 \|f\|_{L_1(\mu)}
  \qquad\text{for all } f\in V,
\end{equation*}
and the argument of Corollary~\ref{cor:near_equal} gives constants $1\pm\tau$ for
$m\ge 34\,\tau^{-2}\gamma n\log(4n/\delta)$. To the best of our knowledge, such a
statement, with i.i.d.\ points, $m\asymp n\log(n/\delta)$, explicit constants and
arbitrary failure probability, has not appeared in the literature before.

Talagrand \cite{Talagrand90} used the Lewis density for the embedding of $n$-dimensional subspaces $V$ of $L_1$ into $\ell_1^m$. A similar technique  underlies the existence result \cite[Cor.~1.1]{DaKoTe23}
in the sampling discretization literature. Sampling from the Lewis
density is treated in \cite[Cor.~4.14]{DaKoTe26} with a worse bound on the number of samples $m$, namely 
$m\gtrsim n\log^3 n$ points and failure probability $n^{-A}$. The same asymptotic bound was proved by Bourgain, Lindenstrauss, and Milman~\cite[Thm.~5.2]{BLM89} already in 1989. In the discrete setting, Cohen and Peng \cite{CohenPeng15} proved a corresponding
result for sampling by $\ell_1$-Lewis weights. Their bounds depend on the number of
rows in a matrix.

Our proof builds on the approach of \cite{Talagrand90},
combined with moment bounds via the Lewis basis. A particular feature is that it
controls the signed error $Z_{\mathrm{sign}}$ simultaneously due to the control of the failure probabilities. The latter does not
follow from the discretization inequality itself, but it is exactly what
Lemma~\ref{lem:box} requires for the weight adaptation in the proof of Theorem \ref{thm:Prob_Tchak}. 
\end{remark}

	\subsection{Richter-Tchakaloff quadrature with high probability}
	
	Recall the definition of the vectors $\mathbf{h}(y)$ in~\eqref{eq:DefH}, associated with the Lewis basis, and
	let us define $\boldsymbol{\varepsilon}\coloneqq \sum_{i=1}^m w_i^0\mathbf h(y_i)- \int_{\Omega} \mathbf h\dint\mu\in\R^n$, which represents the error functional of $Q^0(Y)$ on $V$.
	Indeed, it holds for all $\mathbf{a}\in\R^n$ that
	\begin{equation*}
		\langle \mathbf{a},\boldsymbol{\varepsilon}\rangle=\sum_{i=1}^m w_i^0f_{\mathbf a}(y_i)-\int_\Omega f_{\mathbf{a}}\dint\mu.
	\end{equation*}
	The next lemma provides a criterion for when the weights $w_i^0$ of the quadrature rule $Q^0(Y)$ can be adapted to compensate the quadrature error on $V$. The subsequent duality arguments for positive quadrature have a long history. Mhaskar, Narcowich, and Ward combined the Hahn-Banach and Krein-Rutman theorems to pass
from Marcinkiewicz-Zygmund inequalities to positive exact rules on the sphere, see \cite[Prop.\ 4.1]{MhNaWa01} and the corrigendum \cite{MNW02_corr}. Two-sided weight estimates were obtained by Narcowich, Petrushev, and Ward \cite{NPW06}, Petrushev and Xu \cite[Prop.\ 5.6]{PX08}, and Dai and Wang \cite[Lem.\ 3.4]{DaiWang10}. Lemma~\ref{lem:box} below is a quantitative refinement of this approach. It characterizes exactly when the reference weights $w^0_i$ can be corrected within the box $[(1-\tau)w^0_i, (1+\tau)w^0_i]$, so that exactness, positivity and two-sided weight bounds are obtained simultaneously from a single inequality.

\begin{lemma}\label{lem:box}
		Let $\tau\in[0,1]$.
		There is a weight vector $\mathbf{w}=(w_1,\ldots,w_m)\in\R^m$ with $(1-\tau)w_i^0\le w_i\le(1+\tau)w_i^0$ for all $i=1,\ldots,m$ and $\sum_{i=1}^m w_i\mathbf h(y_i) = \int_{\Omega} \mathbf h\dint\mu$ if and only if
		\begin{equation}\label{eq:boxcrit}
			|\langle \mathbf{a},\boldsymbol{\varepsilon}\rangle|\le\tau \sum_{i=1}^m w_i^0 |f_{\mathbf{a}}(y_i)|
			\qquad\text{for all }\mathbf{a}\in\R^n.
		\end{equation}
	\end{lemma}
	\begin{proof}
		Writing $\mathbf w=\mathbf w^0+\boldsymbol\delta$, the box constraints on $\mathbf w$ read
		$\boldsymbol\delta\in A_\tau\coloneqq \bigtimes_{i=1}^{m}[-\tau w_i^0,\tau w_i^0]$, and the error condition reads $\boldsymbol L\boldsymbol\delta=-\boldsymbol\varepsilon$ with $\boldsymbol{L}\coloneqq (\mathbf h(y_1),\ldots,\mathbf h(y_m))\in\R^{n\times m}$.
		Hence, the existence of $\mathbf w$ is equivalent to $-\boldsymbol\varepsilon\in B_\tau\coloneqq\boldsymbol L A_\tau$.
		The set $B_\tau\subset\R^n$ is convex and compact, so that $\boldsymbol v\in B_\tau$ if and only if
		$\langle\mathbf a,\boldsymbol v\rangle\le s_{B_\tau}(\mathbf a)\coloneqq\sup_{\boldsymbol v'\in B_\tau}\langle\mathbf a,\boldsymbol v'\rangle$ for all $\mathbf a\in\R^n$.
		The support function is given by
		\begin{align*}
			s_{B_\tau}(\mathbf{a})&= \max_{\boldsymbol{\delta}\in A_\tau} \langle \mathbf{a}, \boldsymbol{L}\boldsymbol{\delta} \rangle = \sum_{i=1}^m \max_{\delta_i\in [-\tau w_i^0,\tau w_i^0]} \delta_i\langle \mathbf{a},\mathbf{h}(y_i)\rangle
			=\tau \sum_{i=1}^m w_i^0 |f_{\mathbf{a}}(y_i)|.
		\end{align*}
		Since $s_{B_\tau}$ is even, $-\boldsymbol{\varepsilon}\in B_\tau$ is equivalent to~\eqref{eq:boxcrit}.
	\end{proof}
	
	\begin{proof}[Proof of Theorem~\ref{thm:Prob_Tchak}]
		Let $m\ge67\gamma n$, and assume that the event of Proposition~\ref{prop:conc} occurs, which happens with probability at least $1-4n\exp(-m/(134\gamma n))$.
		Then, we have $Z_{\rm sign}\le\tfrac14$ and $Z_{\rm abs}\le\tfrac12$, see~\eqref{eq:conc}.
		It follows for all $f\in V$ with $\|f\|_{L_1(\mu)}=1$ that
		\begin{equation}
			\label{eq:symcrit_0}
			\biggl|\sum_{i=1}^{m}w_i^0f(y_i)-\int_\Omega f\dint\mu\biggr| \le \frac14
			= \frac12\Big(\int_\Omega|f|\dint\mu-\frac12\Big)
			\le \frac12\sum_{i=1}^{m}w_i^0|f(y_i)| .
		\end{equation}
		By homogeneity, \eqref{eq:symcrit_0} extends to all $f\in V$, which is~\eqref{eq:boxcrit} with $\tau=\frac12$.
		Hence, Lemma~\ref{lem:box} guarantees the existence of weights  $\mathbf{w}\coloneqq (w_1,\ldots,w_m)$ with
		\begin{equation*}
			\frac{1}{2m\varrho(y_i)}\le w_i\le \frac{3}{2m\varrho(y_i)},\qquad i=1,\ldots,m,
		\end{equation*}
		such that the quadrature rule $Q_{\mathbf{w}}(Y)f = \sum_{i=1}^m w_if(y_i)$ is exact on $V$.
		By the reduction in Remark~\ref{rem:reduction}, $Q_{\mathbf w}(Y)$ is then also exact on the original space $V$, since $Y\cap\Omega_N=\emptyset$ holds almost surely and thus does not affect the probability bound.
		If $\mathbf{1}_{\Omega}\in V$ and $\mu(\Omega)=1$, exactness yields $\sum_{i=1}^m w_i=Q_{\mathbf{w}}(Y)\mathbf{1}_{\Omega}=\int_\Omega\mathbf{1}_{\Omega}\dint\mu=1$.
	\end{proof}
	
	\begin{corollary}[Almost equal weights]
		\label{cor:near_equal}
		In the setting of Theorem~\ref{thm:Prob_Tchak}, let $\tau\in(0,\frac12]$ and $\delta\in(0,1)$.
		If
		\begin{equation}
        \label{eq:cor_2.13}
			m\ge \frac{34}{\tau^2}\gamma n\log\Bigl(\frac{4n}{\delta}\Bigr),
		\end{equation}
		then, with probability at least $1-\delta$, there exist weights $\mathbf w=(w_1,\ldots,w_m)$ with
		\begin{equation*}
			\frac{1-\tau}{m\varrho(y_i)}\le w_i\le\frac{1+\tau}{m\varrho(y_i)},\qquad i=1,\ldots,m,
		\end{equation*}
		such that $Q_{\mathbf w}(Y)$ is a Richter-Tchakaloff rule for $V$.
	\end{corollary}
	\begin{proof}
		As in the proof of Proposition~\ref{prop:conc}, Lemma~\ref{lem:moments} and Proposition~\ref{propA:MomTail} imply that $\Prob(Z_\bullet\ge {\ee}^{1/2}\alpha_\bullet u)\le 2n {\ee}^{-u^2/2}$ for $u\in[1,\sqrt{m/(64\gamma n)}]$, where $\alpha_\bullet\coloneqq c_\bullet{\ee}^{-11/23}\sqrt{\gamma n/m}$, see~\eqref{eq:fail_prob_2_40}.
		Now we choose $u\coloneqq\sqrt{2\log(4n/\delta)}$, which satisfies $u<\sqrt{m/(64\gamma n)}$ since $m\ge 136\,\gamma n\log(4n/\delta)$ due to~\eqref{eq:cor_2.13} and $\tau\le1/2$.
		Then both $Z_{\rm sign}<{\ee}^{1/2}\alpha_{\rm sign}u$ and $Z_{\rm abs}<{\ee}^{1/2}\alpha_{\rm abs}u$ hold with probability at least $1-4n{\ee}^{-u^2/2}=1-\delta$.
		By the assumption on $m$, we have ${\ee}^{1/2}\alpha_{\rm sign}u=(8\,{\ee}^{1/23}\gamma n\log(4n/\delta)/m)^{1/2}\le\tau/2$ and ${\ee}^{1/2}\alpha_{\rm abs}u=2\,{\ee}^{1/2}\alpha_{\rm sign}u\le\tau\le\frac12$, using $32\,{\ee}^{1/23}\le 34$.
		Hence, for all $f\in V$ with $\|f\|_{L_1(\mu)}=1$, as in~\eqref{eq:symcrit_0},
		\begin{equation*}
			\biggl|\sum_{i=1}^{m}w_i^0f(y_i)-\int_\Omega f\dint\mu\biggr|\le\frac{\tau}{2}=\tau\Bigl(\int_\Omega|f|\dint\mu-\frac12\Bigr)\le\tau\sum_{i=1}^{m}w_i^0|f(y_i)|,
		\end{equation*}
		which gives~\eqref{eq:boxcrit}.
		Lemma~\ref{lem:box} concludes the proof.
	\end{proof}

	\section{Positive quadrature in RKHS from random nodes}\label{sec:RKHS}
	
	Let us now turn to the RKHS setting specified in the introduction, whereby we assume $\mu(\Omega)>0$. First recall the crucial property of a reproducing kernel Hilbert space $\cH$ that all Dirac functionals $\delta_x$ with $f\mapsto f(x)$ on $\cH$ are continuous.
	Equivalently, the reproducing property $f(x) = \langle f, K(\cdot,x) \rangle_{\cH}$ holds for all $x \in \Omega$ and $f\in\cH$.
	For any orthonormal basis $\{e_j\}_{j\in J}$ of $\cH$, one further has the decomposition
	\begin{align}
		\label{Mercer}
		K(x,y)= \sum_{j\in J} e_j(x)e_j(y)
	\end{align}
	pointwise absolutely, as well as with unconditional convergence in $\cH$ for each function $K(x,\cdot)\in\cH$ or $K(\cdot,y)\in\cH$ on the left of equation~\eqref{Mercer}. 
	Since $\Id_{K,\mu}\colon  \cH \to L_2(\Omega,\mu)$ is Hilbert-Schmidt, it admits an at most countable singular system $\{(\eta_j,\sigma_j,e_j)\}_{j\in J}$, where $\{\eta_j\}_{j \in J}$ and $\{e_j\}_{j \in J}$ with $J \subset \N$ are orthonormal bases (ONB) of the closure of the range of $\Id_{K,\mu}$ and $\ker(\Id_{K,\mu})^\perp$, respectively.
	Further, let $\{\epsilon_k\}_{k \in M}$ be an ONB 
    of $\ker(\Id_{K,\mu})$, which is countable since $\cH$ is separable.
	In view of~\eqref{Mercer}, we thus have the expansion
	\begin{align*}
		K(x,y) = \sum_{j \in J} e_j(x)e_j(y) + \sum_{k \in M} \epsilon_k(x)\epsilon_k(y) .
	\end{align*}
	For each $n\in\N_0$, we define the so-called tail kernel
	\begin{equation*}
		K_n(x,y)
		\coloneqq   \sum_{j>n}
		e_j(x)e_j(y)
		=
		\sum_{j>n}
		\sigma_j^2
		\eta_j(x)\eta_j(y),\qquad x,y\in\Omega,
	\end{equation*}
	for which the associated tail diagonal and tail trace are given by
	\begin{align}
		\label{eqdef:En_and_Tn}
		E_n(x)
		\coloneqq  K_n(x,x)=
		\sum_{j>n}\sigma_j^2\eta_j(x)^2 \quad,
		\quad
		T_n
		\coloneqq 
		\|E_n\|_{L_1(\mu)}
		=
		\sum_{j>n}\sigma_j^2 .
	\end{align}
	Throughout, we fix the representatives $\eta_j \coloneqq \sigma_j^{-1}e_j$, so that these identities hold pointwise.
	We further introduce the subdomain
	\begin{align}
		\label{eqdef:restr_domain_D}
		D \coloneqq  \big\{ x\in\Omega :  K(x,x)= K_0(x,x) \big\} ,
	\end{align}
	whose points are precisely those where $\epsilon_k(x)$ vanishes for all $k\in M$.
	The subdomain $D$ covers almost all of $\Omega$.
	Indeed, since $M$ is countable and $\mu(\{\epsilon_k\neq 0\})=0$ for every $k\in M$ due to $\epsilon_k\in \ker(\Id_{K,\mu})$, the complement $\Omega \setminus D$ is a $\mu$-null set.
	Finally, we set $V_0\coloneqq \{0\}\subset\cH$, and for $n\in\N$ we define
	\begin{align}
		\label{eqdef:V_n}
		V_n \coloneqq  \spann_{\mathbb{R}}\{e_1,\ldots,e_n\} =\spann_{\mathbb{R}}\{\eta_1,\ldots,\eta_n\}.
	\end{align}
	If the singular system is finite with $r<\infty$ elements, we set $V_n=V_r$ for all $n\ge r$.

	Our goal, the proof of Theorem~\ref{thm:intro_RKHS}, will be accomplished in the next section. In this section, as an intermediate step, we first prove Theorem~\ref{thm:main2} below. It gives an estimate for the performance of certain positive quadrature rules obtained from randomly drawn nodes in $\Omega$.
	
	The sampling density $\varrho_n\colon \Omega\to\R_+$, $n\in\N$, used for the generation of the node sets $Y\coloneqq (y_1,\ldots,y_m)\in\Omega^m$ is given as a mixture of $3$ components, namely
	\begin{equation}
		\label{mix_density}
		\varrho_n\coloneqq \frac13\Bigl( \frac{1}{\mu(\Omega)}+\varrho_{{\rm L},n}+ \frac{E_n}{T_n} \Bigr) ,
	\end{equation}
	with $E_n$ and $T_n$ as defined in~\eqref{eqdef:En_and_Tn}.
	If $n\ge r\coloneqq\rank(\Id_{K,\mu})$,
	%, whence $E_n\equiv0$ and $T_n=0$, 
	we hereby replace $E_n/T_n$ by $1/\mu(\Omega)$. Note that, setting $\sigma_j\coloneqq0$ for $j>r$, we have $E_n\equiv0$ and $T_n=0$. The middle component $\varrho_{{\rm L},n}$ in~\eqref{mix_density} is the $L_1$-Lewis density in the sense of Section~\ref{sec:RandomRules} for the space $V_n^\prime\coloneqq V_n + \spann\{ \mathbf{1}_{\Omega} \}$, where $\mathbf{1}_{\Omega}$ denotes the $1$-function on $\Omega$. 
	
	Throughout the rest of this section, $\mu$ is assumed to be a probability measure. This is possible without loss of generality. Further, $B(\cH)$ denotes the unit ball in $\cH$.
	
	\begin{theorem}
		\label{thm:main2}
		Let $(\Omega,\Sigma,\mu)$ be a probability space, let $\cH$ be a separable RKHS of real-valued functions on $\Omega$ with finite-trace kernel $K$, and let $n\in\N$ and $s>1$.
		If $m\ge 402\,n$ nodes $Y\coloneqq (y_1,\ldots,y_m)$ are drawn independently according to $\varrho_n \mu$, then, with probability at least
		\begin{equation}
			\label{eq:main2_prob}
			1-8n\exp\Bigl(-\frac{m}{804\,n}\Bigr)-2^{3/4}m^{1-s},
		\end{equation}
		there exist weights $\mathbf{w}\coloneqq (w_1,\ldots,w_m)$ with
		\begin{equation}
			\label{eq:main2_weights}
			\frac{1}{2m\varrho_n(y_i)} \le w_i\le \frac{3}{2m\varrho_n(y_i)}\le\frac{9}{2m},\quad i=1,\ldots,m,\qquad \sum_{i=1}^{m} w_i=1,
		\end{equation}
		such that the quadrature rule $Q_{\mathbf{w}}(Y)$ is exact on $V_n^\prime \supset V_n$ and
		\begin{equation}
			\label{eq:main2_error}
			\sup_{f\in B(\cH)}
			\biggl|
			\int_{\Omega} f\dint\mu-Q_{\mathbf{w}}(Y)f
			\biggr|
			\le
			(1+\sqrt3)\sigma_{n+1}
			+
			3(1+\sqrt5)\sqrt{\frac{s\log m}{m}} \biggl(\sum_{j>n}\sigma_j^2\biggr)^{1/2} .
		\end{equation}
	\end{theorem}
	
	\begin{corollary}
		\label{cor:main2}
		Let $C\ge1$, $\delta\in(0,1)$, and draw a node set $Y\coloneqq (y_1,\ldots,y_m)$ as in Theorem~\ref{thm:main2} of size
		\begin{equation*}
			m\geq
			\max\left\{
			C\,804\,n\log\!\left(\frac{12n}{\delta}\right),
			\frac{6}{\delta}
			\right\} .
		\end{equation*}
		With probability at least $1-\delta$, there then exist weights $\mathbf{w}\coloneqq (w_1,\ldots,w_m)>0$
		with property~\eqref{eq:main2_weights} such that $Q_{\mathbf{w}}(Y)$ is exact on $V_n^\prime$ and
		\begin{equation*}
			\sup_{f\in B(\cH)}
			\biggl|
			\int_{\Omega} f\dint\mu-Q_{\mathbf{w}}(Y)f
			\biggr|
			\le
			(1+\sqrt3) \sigma_{n+1}
			+
			\sqrt{\frac{\log(\ee C)}{C}} \biggl(\frac{1}{ n}\sum_{j>n}\sigma_j^2\biggr)^{1/2} .
		\end{equation*}
	\end{corollary}
	\begin{proof}
		The statement follows from Theorem~\ref{thm:main2}.
		With $s\coloneqq2$, we have
		\begin{align*}
			8n\exp\Bigl(-\frac{m}{804\,n}\Bigr)
			\leq
			\frac{2\delta}{3} 
			\qquad\text{and}
			\qquad
			2^{3/4}m^{1-s}
			\leq
			\frac{2^{3/4}}{6} \delta
			<
			\frac{\delta}{3}. 
		\end{align*}
		Hence, with probability at least $1-\delta$, there indeed exist weights $\mathbf{w}\coloneqq (w_1,\ldots,w_m)$ such that~\eqref{eq:main2_weights} and~\eqref{eq:main2_error} hold true. Since the function $x\mapsto \log(x)/x$ is decreasing on
		$[\ee,\infty)$, we further have
		\begin{align*}
			\frac{s\log m}{m}
			=
			\frac{2\log m}{m}
			&\leq
			\frac{
				2\log\!\left(
				804\,C n\log\!\left(\frac{12n}{\delta}\right)
				\right)
			}{
				804\,C n\log\!\left(\frac{12n}{\delta}\right)
			}\notag
			\\
			&\leq
			\frac{\log(\ee C)}{C}\,
			\frac{
				2\log\!\left(
				804\,n\log\!\left(\frac{12n}{\delta}\right)
				\right)
			}{
				804\,n\log\!\left(\frac{12n}{\delta}\right)
			}
			\leq \frac{\log(\ee C)}{C}\,
			\frac{1}{131\,n} ,
		\end{align*}
		and consequently
		\begin{equation*}
			3(1+\sqrt5)
			\sqrt{\frac{s\log m}{m}}
			\leq
			\frac{3(1+\sqrt5)}{\sqrt{131\,n}}
			\sqrt{\frac{\log(\ee C)}{C}}
			\leq
			\frac{1}{\sqrt n}
			\sqrt{\frac{\log(\ee C)}{C}} .
			\qedhere
		\end{equation*}
	\end{proof}
	
	\noindent
	The proof of Theorem~\ref{thm:main2} proceeds via several steps.
	Since $\varrho_n\ge\frac13\varrho_{{\rm L},n}$, condition~\eqref{eq:rho_admissible} holds for $V\coloneqq V_n^\prime$ with $\gamma\coloneqq 3$.
	Hence, Theorem~\ref{thm:Prob_Tchak} applies with $\dim V_n^\prime\le n+1\le 2n$, i.e., under the condition $m\ge 67\cdot3\cdot2n=402\,n$ and with failure probability at most $4(n+1)\exp(-m/(402(n+1)))\le8n\exp(-m/(804\,n))$.
	It provides the exactness of $Q_{\mathbf{w}}(Y)$ on $V_n^\prime$ and the weight bounds in~\eqref{eq:main2_weights}. The uniform estimate $w_i\le 9/(2m)$ is hereby due to $\varrho_n\ge\frac{1}{3\mu(\Omega)}=\frac13$.
	For the error bound in~\eqref{eq:main2_error}, we express the worst-case error of rules that are exact on $V_n$ in terms of the tail of the singular system, and then estimate the associated tail Gram matrix.
	The latter relies on the component $E_n/T_n$ of $\varrho_n$.
	
	\subsection{An error identity for exact rules on \texorpdfstring{$V_n$}{Vn}}
	
	For $x\in\Omega$ and $n\in\N_0$, let us define the vectors
	\begin{equation}
		\label{eqdef:psi_n}
		\boldsymbol{\psi}_n(x)\coloneqq (\sigma_j\eta_j(x))_{j>n}=(e_j(x))_{j>n}
		\quad\text{and}\quad
		\int_\Omega\boldsymbol{\psi}_n\dint\mu\coloneqq\Bigl(\int_\Omega e_j\dint\mu\Bigr)_{j>n} ,
	\end{equation}
	which are elements of $\ell_2(\N)$.
	Indeed, by~\eqref{eqdef:En_and_Tn} and~\eqref{Mercer}, it holds $\|\boldsymbol{\psi}_n(x)\|_2^2=\sum_{j>n}e_j(x)^2=E_n(x)\le K(x,x)<\infty$.
	Moreover, since $e_j=\sigma_j\eta_j$ in $L_2(\Omega,\mu)$, we have $\int_\Omega e_j\dint\mu=\sigma_j\langle\eta_j,\mathbf{1}_\Omega\rangle_{L_2(\mu)}$, so that Bessel's inequality yields
	\begin{equation}
		\label{eq:mean_tail}
		\biggl\|\int_\Omega\boldsymbol{\psi}_n\dint\mu\biggr\|_2^2
		=\sum_{j>n}\sigma_j^2\langle\eta_j,\mathbf{1}_\Omega\rangle_{L_2(\mu)}^2
		\le\sigma_{n+1}^2\|\mathbf{1}_\Omega\|_{L_2(\mu)}^2=\sigma_{n+1}^2.
	\end{equation}
	As in Section~\ref{sec:RandomRules}, we treat elements of $\ell_2(\N)$ as column vectors.
	
	\begin{proposition}
		\label{prop:err_decomp}
		Let $\mu$ be a probability measure.
		For any quadrature rule $Q_{\mathbf{w}}(X)$ on $\cH$ with nodes $X=(x_1,\ldots,x_m)\in D^m$, see~\eqref{eqdef:restr_domain_D}, and weights $\mathbf{w}=(w_1,\ldots,w_m)\in\R^m$, which is furthermore exact on $V_n$, see~\eqref{eqdef:V_n}, it holds
		\begin{equation}
			\label{eq:err_identity}
			\sup_{f\in B(\cH)}
			\biggl|
			\int_{\Omega} f\dint\mu - Q_{\mathbf{w}}(X)f
			\biggr|
			=
			\biggl\|\int_\Omega\boldsymbol{\psi}_n\dint\mu-\sum_{i=1}^{m} w_i\boldsymbol{\psi}_n(x_i)\biggr\|_2
			\le
			\sigma_{n+1}
			+
			\biggl\|\sum_{i=1}^{m} w_i\boldsymbol{\psi}_n(x_i)\biggr\|_2 .
		\end{equation}
	\end{proposition}
	\begin{proof}
		Let $(e_j)_{j}$ and $(\epsilon_k)_{k}$ be the right singular functions and the kernel basis, respectively, which constitute an orthonormal basis for $\cH$.
		With these functions, we can write $f\in B(\cH)$ as $f=h+t$ with
		\begin{equation*}
			h\coloneqq \sum_{j\le n}c_j e_j \quad\text{and}\quad t\coloneqq \sum_{j>n}c_j e_j + \sum_{k\in M}d_k \epsilon_k
		\end{equation*}
		and coefficients that satisfy $\sum_{j\ge 1} c_j^2 + \sum_{k\in M} d_k^2 \le 1$.
		Note that all series converge in $\cH$.
		Since $Q_{\mathbf{w}}(X)$ is exact on $V_n$ and $h\in V_n$, we deduce that
		\begin{equation*}
			\int_{\Omega} f\dint\mu - Q_{\mathbf{w}}(X)f = \int_{\Omega} t\dint\mu -  Q_{\mathbf{w}}(X)t .
		\end{equation*}
		Since point evaluations and integration with respect to $\mu$ are in $\cH^\prime$ (by the reproducing property and~\eqref{P2:integrab}), both functionals can be applied term by term to the series defining $t$.
		Since the nodes $X=(x_1,\ldots,x_m)$ lie in $D$, we have $\epsilon_k(x_i)=0$, and since $\epsilon_k\in\ker(\Id_{K,\mu})$ vanishes $\mu$-almost everywhere, we have $\int_\Omega\epsilon_k\dint\mu=0$.
		Therefore, taking into account the countability of $M$, it holds
		\begin{equation*}
			t(x_i)=\sum_{j>n}c_j e_j(x_i) = \sum_{j>n}c_j\sigma_j \eta_j(x_i)
			\quad\text{and}\quad
			\int_\Omega t\dint\mu=\sum_{j>n}c_j\int_\Omega e_j\dint\mu .
		\end{equation*}
		With $\mathbf{c} \coloneqq  (c_j)_{j>n}$, we now obtain
		\begin{equation*}
			\int_{\Omega} f\dint\mu - Q_{\mathbf{w}}(X)f
			=
			\biggl\langle\mathbf{c},\int_\Omega\boldsymbol{\psi}_n\dint\mu-\sum_{i=1}^{m} w_i\boldsymbol{\psi}_{n}(x_i)\biggr\rangle .
		\end{equation*}
		Since $\|\mathbf{c}\|_2\le 1$, the Cauchy-Schwarz inequality shows that the left-hand side of~\eqref{eq:err_identity} is at most the middle term.
		Conversely, if $\mathbf{b}\coloneqq\int_\Omega\boldsymbol{\psi}_n\dint\mu-\sum_{i=1}^m w_i\boldsymbol{\psi}_{n}(x_i)\neq\mathbf{0}$, the choice $f\coloneqq\|\mathbf{b}\|_2^{-1}\sum_{j>n}b_je_j\in B(\cH)$ yields equality.
		The estimate in~\eqref{eq:err_identity} follows from the triangle inequality and~\eqref{eq:mean_tail}.
	\end{proof}
	
	\noindent
	Note that Proposition~\ref{prop:err_decomp} holds for arbitrary real weights.
	The weight bounds~\eqref{eq:main2_weights}, which imply positivity, become essential only when bounding the remaining term $\|\sum_{i=1}^m w_i\boldsymbol{\psi}_n(x_i)\|_2$ in Section~\ref{Sec:ProofThmSampled}.
	
	\subsection{Tail Gram estimate}
    %\label{sec:TailBound}
	
	For $y\in\Omega$, let us define the infinite sequences
	\begin{equation}
		\label{tilde_psi(y)}
		\tilde{\boldsymbol{\psi}}_n(y)\coloneqq (\sigma_j\eta_j(y))_{j>n}/\sqrt{\varrho_n(y)}=\boldsymbol{\psi}_n(y)/\sqrt{\varrho_n(y)},
	\end{equation}
	see~\eqref{eqdef:psi_n}, which are elements of $\ell_2(\N)$ due to $\boldsymbol{\psi}_n(y)\in\ell_2(\N)$ and $\varrho_n(y)\ge\frac{1}{3\mu(\Omega)}=\frac{1}{3}$.
	For a node set $Y\coloneqq (y_1,\ldots,y_m)\in\Omega^m$, we
	further define
	\begin{equation*}
		%\label{eqdef:Phi_tilde}
		\tilde{\boldsymbol{\Phi}}_{m,n}\coloneqq  \frac{1}{\sqrt{m}}\left(\begin{array}{rcl}
			- & \tilde{\boldsymbol{\psi}}_{1,n}^{\top} & - \\
			& \vdots       &\\
			- & \tilde{\boldsymbol{\psi}}_{m,n}^{\top} & -
		\end{array}\right)
	\end{equation*}
	with the shorthands $\tilde{\boldsymbol{\psi}}_{i,n}\coloneqq  \tilde{\boldsymbol{\psi}}_n(y_i)$, $i=1,\ldots,m$.
	The tail Gram matrix then has the form
	\begin{equation*}
		\tilde{\boldsymbol{\Phi}}_{m,n}^{\top} \tilde{\boldsymbol{\Phi}}_{m,n} =
		\frac{1}{m} \sum_{i=1}^m\tilde{\boldsymbol{\psi}}_{i,n}\tilde{\boldsymbol{\psi}}_{i,n}^{\top} .
	\end{equation*}
	
	\begin{proposition}
		%[tail Gram under the mixture]
		\label{prop:gram}
		Let $n<r$, $s>1$, and $m\ge 3$.
		If the nodes $Y\coloneqq (y_1,\ldots,y_m)\in\Omega^m$ are drawn independently according to $\varrho_n \mu$, then
		\begin{equation}
			\label{eq:conc-est}
\Prob\biggl(\lambda_{\max}\biggl(\frac1m\sum_{i=1}^m\tilde{\boldsymbol{\psi}}_{i,n}\tilde{\boldsymbol{\psi}}_{i,n}^{\top}\biggr)
			\ge 2\sigma_{n+1}^2+12(3+\sqrt5)\frac{s\log m}{m}T_n\biggr)\le 2^{3/4}m^{1-s} ,
		\end{equation}
		where $\tilde{\boldsymbol{\psi}}_{i,n}\coloneqq \tilde{\boldsymbol{\psi}}_n(y_i)$, see~\eqref{tilde_psi(y)}.
	\end{proposition}
	\begin{proof}
		By the orthonormality of $(\eta_j)_{j\ge1}$ in $L_2(\Omega,\mu)$, the random sequences $\tilde{\boldsymbol{\psi}}_{i,n}$ are i.i.d.\ with
		\begin{equation*}
			\E_{y \sim \varrho_n\mu}\bigl[\tilde{\boldsymbol{\psi}}_{i,n}\tilde{\boldsymbol{\psi}}_{i,n}^{\top}\bigr]=\int_{\Omega} \boldsymbol{\psi}_{n}(y)\boldsymbol{\psi}_{n}(y)^{\top}\dint\mu(y)
			=\mathrm{diag}(\sigma_j^2)_{j>n} \eqqcolon\boldsymbol{\Lambda}_n.
		\end{equation*}
		Further, since $\varrho_n\ge E_n/(3T_n)$ according to \eqref{mix_density}, it holds for every $y\in\Omega$ that
		\begin{equation*}
			\|\tilde{\boldsymbol{\psi}}_n(y)\|_2^2=\frac{E_n(y)}{\varrho_n(y)}\le3T_n .
		\end{equation*}
		Hence, we can apply Theorem~\ref{thm:MoUl20} with
		$L^2\coloneqq 3T_n$, $\|\boldsymbol{\Lambda}_n\|_{2\to2}=\sigma_{n+1}^2$, and the parameter $s>1$, whose exceptional event has probability at most $2^{3/4}m^{1-s}$.
		Moreover, since $8\cdot3\cdot\bigl(\frac{1+\sqrt5}{2}\bigr)^2=12(3+\sqrt5)$, the threshold in Theorem~\ref{thm:MoUl20} is
		\begin{equation*}
			F=\max\Bigl\{
			12(3+\sqrt5)\frac{s\log m}{m}T_n,
			\sigma_{n+1}^2
			\Bigr\}.
		\end{equation*}
		Outside the exceptional event, we thus have
		\begin{align*}
        \lambda_{\max}\biggl(\frac1m\sum_{i=1}^m\tilde{\boldsymbol{\psi}}_{i,n}\tilde{\boldsymbol{\psi}}_{i,n}^{\top}\biggr)
			&\le \|\boldsymbol{\Lambda}_n\|_{2\to2}
			+\biggl\|\frac1m\sum_{i=1}^m
			\tilde{\boldsymbol{\psi}}_{i,n}\tilde{\boldsymbol{\psi}}_{i,n}^{\top}-\boldsymbol{\Lambda}_n\biggr\|_{2\to2}\\
			&<\sigma_{n+1}^2+F
			\le2\sigma_{n+1}^2
			+12(3+\sqrt5)\frac{s\log m}{m}T_n,
		\end{align*}
		which is precisely~\eqref{eq:conc-est}.
	\end{proof}
	
	\subsection{Proof of Theorem~\ref{thm:main2}}
	\label{Sec:ProofThmSampled}
	
	%\begin{proof}[Proof of Theorem~\ref{thm:main2}]
	Assume first $n<r$.
	By Theorem~\ref{thm:Prob_Tchak}, applied with $V\coloneqq V_n^\prime$, $\varrho\coloneqq\varrho_n$, and $\gamma\coloneqq 3$, there exist, with probability at least $1-8n\exp(-m/(804\,n))$, see the discussion after Theorem~\ref{thm:main2}, weights $\mathbf{w}\coloneqq (w_1,\ldots,w_m)$ with $\frac12 w_i^0\le w_i\le\frac32 w_i^0$ such that $Q_{\mathbf{w}}(Y)$ is exact on $V_n^\prime$.
	As $\mathbf{1}_\Omega\in V_n^\prime$ and $\mu(\Omega)=1$, this gives $\sum_{i=1}^m w_i=1$. Since $\varrho_n\ge\frac{1}{3\mu(\Omega)}=\frac13$, we also get $0\le w_i\le 3/(2m\varrho_n(y_i)) \leq 9/(2m)$.
	By Proposition~\ref{prop:gram}, which applies since $m\ge402n\ge3$, the event in~\eqref{eq:conc-est} occurs with probability at most $2^{3/4}m^{1-s}$.
	Moreover, the randomly drawn nodes lie in $D$ almost surely, as $\mu(\Omega\setminus D)=0$.
	Hence, with probability at least~\eqref{eq:main2_prob}, all three properties hold simultaneously.
	Since $Q_{\mathbf{w}}(Y)$ is exact on $V_n\subset V_n^\prime$ and $Y\subset D$, Proposition~\ref{prop:err_decomp} yields, with $\boldsymbol{\psi}_{i,n}\coloneqq\boldsymbol{\psi}_n(y_i)$, that
	\begin{equation}
		\label{eq:proof_main2_1}
		\sup_{f\in B(\cH)} \biggl| \int_{\Omega} f\dint\mu-Q_{\mathbf{w}}(Y)f \biggr|
		\le \sigma_{n+1} + \biggl\|\sum_{i=1}^m w_i\boldsymbol{\psi}_{i,n}\biggr\|_2 .
	\end{equation}
	To estimate the last term, we set
	\begin{equation*}
		%\label{eqdef:Psi_{m,n}}
		\boldsymbol{\Psi}_{m,n} \coloneqq   \left(\begin{array}{rcl}
			- & \sqrt{w_1} \boldsymbol{\psi}_{1,n}^{\top} & - \\
			& \vdots       &\\
			- & \sqrt{w_m} \boldsymbol{\psi}_{m,n}^{\top} & -
		\end{array}\right) .
	\end{equation*}
	Then, it holds $\sum_{i=1}^m w_i\boldsymbol{\psi}_{i,n}=\boldsymbol{\Psi}_{m,n}^{\top}(\sqrt{w_1},\ldots,\sqrt{w_m})^{\top}$, where $w_i\ge 0$ is essential.
	From this, we obtain with $\sum_{i=1}^m w_i=1$ that
	\begin{equation}
		\label{eq:proof_main2_2}
		\biggl\|\sum_{i=1}^m w_i\boldsymbol{\psi}_{i,n}\biggr\|_2^2
		\le\biggl(\sum_{i=1}^m w_i\biggr)\|\boldsymbol{\Psi}_{m,n}\|_{2\to2}^2
		= \lambda_{\max}\bigl(\boldsymbol{\Psi}_{m,n}^{\top}\boldsymbol{\Psi}_{m,n}\bigr) \le \frac32 \lambda_{\max}\bigl(\tilde{\boldsymbol{\Phi}}_{m,n}^{\top}\tilde{\boldsymbol{\Phi}}_{m,n}\bigr) ,
	\end{equation}
	where the last step is based on the estimate
	\begin{equation*}
		\boldsymbol{\Psi}_{m,n}^{\top}\boldsymbol{\Psi}_{m,n} =\sum_{i=1}^m w_i \boldsymbol{\psi}_{i,n}\boldsymbol{\psi}_{i,n}^{\top} \preceq \frac{3}{2m} \sum_{i=1}^m  \tilde{\boldsymbol{\psi}}_{i,n}\tilde{\boldsymbol{\psi}}_{i,n}^{\top} = \frac32\tilde{\boldsymbol{\Phi}}_{m,n}^{\top}\tilde{\boldsymbol{\Phi}}_{m,n}.
	\end{equation*}
	Note that these estimates hold for every weight vector satisfying~\eqref{eq:main2_weights}, so it does not matter that $\mathbf{w}$ depends on $Y$.
	Invoking~\eqref{eq:proof_main2_1}, \eqref{eq:proof_main2_2}, and Proposition~\ref{prop:gram}, we finally arrive at
	\begin{align*}
		\sup_{f\in B(\cH)}
		\biggl|\int_{\Omega} f\dint\mu-Q_{\mathbf{w}}(Y)f\biggr|
		&\le \sigma_{n+1}
		+\Bigl(3\sigma_{n+1}^2
		+18(3+\sqrt5)\frac{s\log m}{m}T_n\Bigr)^{1/2}\notag\\
		&\le (1+\sqrt3)\sigma_{n+1}
		+3(1+\sqrt5)\sqrt{\frac{s\log m}{m}}\sqrt{T_n},
	\end{align*}
	where we used $18(3+\sqrt5)=(3(1+\sqrt5))^2$.
	If $n\ge r$, then $\boldsymbol{\psi}_n\equiv\mathbf{0}$, since there are no indices $j>n$ in~\eqref{eqdef:psi_n}.
	Hence, Proposition~\ref{prop:err_decomp} shows that every rule with nodes in $D$ that is exact on $V_n$ is exact on $\cH$.
	Besides $Y \subset D$, which holds almost surely, only the event of Theorem~\ref{thm:Prob_Tchak} is needed, which occurs with probability at least $1-8n\exp(-m/(804 n))$.
    The weight bounds follow as above. \hfill\qed
	%\end{proof}

	\section{Compression to constant oversampling}
	\label{sec:compression}
	
	In this section, we reduce the number of nodes for the rules from Theorem~\ref{thm:main2} from $m\asymp n\log(\ee n)$ to $O(n)$, which completes the proof of Theorem~\ref{thm:main1}.
	The argument is constructive.
	
	\subsection{A constructive compression lemma}
	The compression lemma below has a similar flavor as the subsampling procedures used by Nagel, Sch\"afer, T. Ullrich \cite{NaSchUl22} and Dolbeault, Krieg, and M.~Ullrich \cite[Prop.~17]{DoKrUl23}. For a constructive version, see Bartel, Sch\"afer, Ullrich \cite{BaSchUl23} for least squares recovery based on sparsification results by Batson, Spielman, Srivastava \cite{BSS12}. However, the compression step here is more elementary. 
	
	\begin{lemma}
		\label{lem:compress}
		Let $\mathbf{u}_1,\ldots,\mathbf{u}_m\in\ell_2(\N)$ with $\|\mathbf{u}_i\|_2\le R$ for some $R\ge0$, and let $\mathbf{p}=(p_1,\ldots,p_m)$ with $p_i\ge0$ and $\sum_{i=1}^m p_i=1$. Further, let $d\ge1$ and $\mathbf{L}\in\R^{d\times m}$ be a matrix whose row space contains the vector $(1,\ldots,1)\in\R^m$. Then, for every $k\in\N$, there is a vector $\mathbf{p}'=(p_1',\ldots,p_m')$ with $\max\{p_i,1/k \}\ge p_i'\ge 0$ such that
		\begin{equation*}
			%\label{eq:compress}
		\mathbf{L}\mathbf{p}'=\mathbf{L}\mathbf{p},\qquad
			\bigl|\{i: p_i'\neq0\}\bigr|\le k+d-1,\qquad
			\biggl\|\sum_{i=1}^m(p_i'-p_i)\mathbf{u}_i\biggr\|_2^2\le\frac{R^2}{k}\Bigl(1+\frac{d}{k}\Bigr).
		\end{equation*}
	\end{lemma}
	
	\begin{proof}
		First note that one can just choose $\mathbf{p}'\coloneqq \mathbf{p}$ if not more than $k+d-1$ entries of $\mathbf{p}$ are nonzero.
		Otherwise, we deploy an iterative procedure that reduces $\mathbf{p}$ to a vector $\mathbf{p}'$ with the desired properties.
		To this end, set $I\coloneqq\{i:p_i<1/k\}$, and define the nonempty polytope
		\begin{equation*}
			\mathcal P\coloneqq \bigl\{\mathbf{v}\in\R^m:\mathbf{L}\mathbf{v}=\mathbf{L}\mathbf{p},\ 0\le v_i\le 1/k \text{ for } i\in I,\ v_i=p_i \text{ for } i\notin I\bigr\},
		\end{equation*}
		and the linear map $\mathbf{U}\colon\R^m\to\ell_2(\N)$ with $\mathbf{v}\mapsto\sum_{i=1}^m v_i\mathbf{u}_i$.
		
		The idea for the iteration is simple. Starting with $\mathbf{v}^{(0)}\coloneqq\mathbf{p}$, one constructs points $\mathbf{v}^{(t)}\in\mathcal P$, $t\in\{1,\ldots,T\}$, by moving in each step along a unit direction $\mathbf{d}^{(t)}\in\ker\mathbf L$ until at most $k+d-1$ coordinates are nonzero.
		Let $\mathbf{v}^{(t)}$ be the current point, \smash{$J_t\coloneqq \{i\in I:0<v_i^{(t)}<1/k\}$} the set of its free coordinates, and suppose that $\mathbf v^{(t)}$ has at least $k+d$ nonzero entries.
		The free coordinates carry the mass \smash{$M_t\coloneqq\sum_{i\in J_{t}}v_i^{(t)}$}, and all other nonzero coordinates are at least $1/k$.
		As \smash{$\sum_{i=1}^m v_i^{(t)}=1$} by the assumption on $\mathbf L$, there are at most $k(1-M_t)$ entries of at least $1/k$, so that $k+d\le k(1-M_t)+|J_t|$, i.e., $|J_t|-d\ge kM_t$.
		In particular, $J_t\neq\emptyset$ and hence $M_t>0$.
		Since $\rank(\mathbf{L})\le d$, the space $\mathcal{K}_t\coloneqq\{\mathbf c\in\ker\mathbf{L}: c_i=0 \text{ for } i\notin J_t\}$ of directions that only move free coordinates has dimension $\nu_t\ge|J_t|-d\ge kM_t>0$.
		
		Let $\mathbf{d}^{(t)}\in\mathcal{K}_t$ be a unit vector minimizing $\|\mathbf{U}\mathbf{d}\|_2^2$ among all unit vectors $\mathbf{d}\in\mathcal{K}_t$, and set $\lambda_t\coloneqq\max\{R^2,\|\mathbf{U}\mathbf{d}^{(t)}\|_2^2\}$.
		The set of all $\alpha\in\R$ with $\mathbf{v}^{(t)}+\alpha\mathbf{d}^{(t)}\in\mathcal P$ is a bounded closed interval with $0$ in its interior, since $\mathbf{d}^{(t)}\in\ker\mathbf{L}$ is nonzero only for coordinates in $J_t$.
		On this interval, 
		\begin{equation*}
			\alpha\mapsto \|\mathbf{U}(\mathbf{v}^{(t)}+\alpha\mathbf{d}^{(t)}-\mathbf{p})\|_2^2-\lambda_t\|\mathbf{v}^{(t)}+\alpha\mathbf{d}^{(t)}\|_2^2
		\end{equation*}
		is a concave function in $\alpha$ since its leading coefficient $\|\mathbf{U}\mathbf{d}^{(t)}\|_2^2-\lambda_t\|\mathbf{d}^{(t)}\|_2^2$ is nonpositive.
		Hence, at one of the two endpoints, its value does not exceed its value at $0$, and we let $\mathbf{v}^{(t+1)}$ be such an endpoint, where
		\begin{equation}
			\label{eq:compress_step}
			\|\mathbf{U}(\mathbf{v}^{(t+1)}-\mathbf{p})\|_2^2\le \|\mathbf{U}(\mathbf{v}^{(t)}-\mathbf{p})\|_2^2 + \lambda_t\bigl(\|\mathbf{v}^{(t+1)}\|_2^2-\|\mathbf{v}^{(t)}\|_2^2\bigr).
		\end{equation}
		There, at least one previously free coordinate has reached one of its bounds $0$ or $1/k$, while all coordinates outside of $J_t$ remain unchanged.
		We proceed with this iteration, where clearly $J_{t+1}\subsetneq J_t$ and $\mathcal K_{t+1}\subset\mathcal K_t$, until a vector $\mathbf{p}'\coloneqq \mathbf{v}^{(T)}$ is reached with less than $k+d$ nonzero entries. At the latest, this is the case when $J_t=\emptyset$.
		In particular, $(\lambda_t)$ is nondecreasing.
		
		Set $S_t\coloneqq\|\mathbf{p}'\|_2^2-\|\mathbf{v}^{(t)}\|_2^2$, so that $S_T=0$.
		Since $\mathbf{p}'$ and $\mathbf{v}^{(t)}$ agree outside of $J_t$, $\sum_{i\in J_t}p_i'=M_t$, and $0\le p_i'\le 1/k$ for $i\in J_t\subset I$, we have $S_t\le\sum_{i\in J_t}p_i'^2\le M_t/k\le\nu_t/k^2$, and likewise $S_0\le\sum_{i\in I}p_i'^2\le1/k$.
		Summing~\eqref{eq:compress_step} over $t<T$ and using partial summation with $\lambda_{-1}\coloneqq R^2$ as well as $\lambda_t\ge\lambda_{t-1}$, we obtain
		\begin{equation}
			\label{eq:compress_sum}
			\|\mathbf{U}(\mathbf{p}'-\mathbf{p})\|_2^2\le \sum_{t=0}^{T-1}\lambda_t\bigl(S_t-S_{t+1}\bigr)
			= R^2S_0+\sum_{t=0}^{T-1}(\lambda_t-\lambda_{t-1})S_t
			\le \frac{R^2}{k}+\frac{1}{k^2}\sum_{t=0}^{T-1}(\lambda_t-\lambda_{t-1})\nu_t .
		\end{equation}
		If $\lambda_t=R^2$ for all $t<T$, we are done.
		Otherwise, let $t_0$ be the first $t$ with $\lambda_t>R^2$, so that only the terms with $t\ge t_0$ remain. Noting $\lambda_{t_0-1}-R^2=0$, again with partial summation 
		\begin{equation}
			\label{eq:compress_tail_0}
			\sum_{t=t_0}^{T-1}(\lambda_t-\lambda_{t-1})\nu_t
			%=\sum_{t=t_0}^{T-1}\bigl((\lambda_t-R^2)-(\lambda_{t-1}-R^2)\bigr)\nu_t
			=(\lambda_{T-1}-R^2)\nu_{T-1} + \sum_{t=t_0}^{T-2}(\lambda_t-R^2)(\nu_t-\nu_{t+1}) .
		\end{equation}
		By nestedness, we can choose an orthonormal basis $\mathbf b_1,\ldots,\mathbf b_{\nu_{t_0}}$ of $\mathcal K_{t_0}$ such that $\mathbf b_1,\ldots,\mathbf b_{\nu_t}$ is a basis of $\mathcal K_t$ for $t_0\le t<T$.
		For $l\in\{1,\ldots,\nu_{t_0}\}$, let $r_l$ be the largest $t\in\{t_0,\ldots,T-1\}$ with $\mathbf b_l\in\mathcal K_t$. Then $\lambda_{r_l}=\|\mathbf U\mathbf{d}^{(r_l)}\|_2^2\le\|\mathbf U\mathbf b_l\|_2^2$ by the minimality of $\mathbf d^{(r_l)}$, and $\nu_t$ is the number of $l$ with $r_l\ge t$.
		Consequently, 
		\begin{equation}
			\label{eq:compress_tail}
			%\sum_{t=t_0}^{T-1}(\lambda_t-\lambda_{t-1})\nu_t=
			(\lambda_{T-1}-R^2)\nu_{T-1} + \sum_{t=t_0}^{T-2}(\lambda_t-R^2)(\nu_t-\nu_{t+1})=\sum_{l=1}^{\nu_{t_0}}\bigl(\lambda_{r_l}-R^2\bigr)
			\le\sum_{l=1}^{\nu_{t_0}}\|\mathbf U\mathbf b_l\|_2^2-\nu_{t_0}R^2 .
		\end{equation}
		To bound the right-hand side, we write $\mathbf u_i=(u_{i,s})_{s\in\N}$.
		Since $\mathbf b_l$ vanishes outside of $J_{t_0}$, we may regard $\mathbf b_1,\ldots,\mathbf b_{\nu_{t_0}}$ as orthonormal vectors in $\R^{J_{t_0}}$, and the $s$-th coordinate of $\mathbf U\mathbf b_l$ is $\langle\mathbf b_l,(u_{i,s})_{i\in J_{t_0}}\rangle$.
		Hence, Bessel's inequality, together with $\|\mathbf u_i\|_2\le R$ and $\nu_{t_0}\ge|J_{t_0}|-d$, yields
		\begin{equation*}
			\sum_{l=1}^{\nu_{t_0}}\|\mathbf U\mathbf b_l\|_2^2
			=\sum_{s\in\N}\sum_{l=1}^{\nu_{t_0}}\bigl\langle\mathbf b_l,(u_{i,s})_{i\in J_{t_0}}\bigr\rangle^2
			\le\sum_{s\in\N}\sum_{i\in J_{t_0}}u_{i,s}^2
			=\sum_{i\in J_{t_0}}\|\mathbf u_i\|_2^2
			\le(\nu_{t_0}+d)R^2 .
		\end{equation*}
		Thus, \eqref{eq:compress_tail} and~\eqref{eq:compress_tail_0} are bounded by $dR^2$, and inserting this into~\eqref{eq:compress_sum} proves the claim.
	\end{proof}

\noindent
In combination with Theorem~\ref{thm:Prob_Tchak}, we get a result on Richter-Tchakaloff rules with bounded weights.

\begin{theorem}
		\label{thm:RT_bounded}
		Let $V\subset L_1(\Omega,\mu)$ be a space of real-valued functions on $\Omega$ with $\dim V=n$, and assume that $\mu$ is a probability measure. Then there exist nodes $X=(x_1,\ldots,x_{2n})\in\Omega^{2n}$ and weights
        \begin{equation*}
			0\le w_i\le\frac{3}{n},\qquad i=1,\ldots,2n,
		\end{equation*}
        such that $Q_{\mathbf{w}}(X)$ is a Richter-Tchakaloff rule for $V$, i.e., $Q_{\mathbf{w}}(X)f=\int_\Omega f\dint\mu$ for all $f\in V$.
	\end{theorem}
    \begin{proof}
    We first apply Theorem~\ref{thm:Prob_Tchak} with the probability density 
    \begin{equation*}
		\varrho\coloneqq \frac12\bigl( 1+\varrho_{\rm L} \bigr) ,
	\end{equation*}
    where $\varrho_{\rm L}$ denotes the $L_1$-Lewis density of $V$. The density $\varrho$ satisfies~\eqref{eq:rho_admissible} with $\gamma=2$.
    For $m\in\N$ large enough, we obtain a Richter-Tchakaloff rule $Q_{\mathbf{w}'}(Y)$ with nodes $Y=(y_1,\ldots,y_m)$  and positive weights $\mathbf{w}'=(w'_1,\ldots,w'_m)$ such that $ \frac{1}{2m\varrho(y_i)}\le w'_i \le \frac{3}{2m\varrho(y_i)}$.
    Recall that the nodes are hereby drawn independently according to $\varrho\mu$.
Next, for compression, we apply Lemma~\ref{lem:compress} with $k\coloneqq n$ and $d\coloneqq n+1$. We let $\mathbf p$ be the vector with entries $p_i\coloneqq w'_i\varrho(y_i)/\beta$, where $\beta\coloneqq \sum_{i=1}^m w_i'\varrho(y_i)\in(0,3/2]$. As matrix $\mathbf{L}$, we use $\mathbf{L}\in\R^{d\times m}$ with first row $(1,\ldots,1)$. The remaining rows are chosen as $(f_j(y_i)/\varrho(y_i))_{i=1}^m$ for some basis $f_1,\ldots,f_{n}$ of $V$.
		Lemma~\ref{lem:compress} then yields a vector $\mathbf{p}'=(p_1',\ldots,p_m')$, which we use to define $w_i\coloneqq \beta p_i'/\varrho(y_i)\ge0$. The node set $X$ shall consist of those nodes $y_i$ where $w_i\neq0$.
		Then $|X|\le k+d-1 = 2n$. Further, due to $\mathbf{L}\mathbf{p}'=\mathbf{L}\mathbf{p}$, 
		\begin{equation*}
			\sum_{i=1}^m w_i f_j(y_i)=\beta\sum_{i=1}^m p_i'\frac{f_j(y_i)}{\varrho(y_i)}
			=\beta\sum_{i=1}^m p_i\frac{f_j(y_i)}{\varrho(y_i)}=\sum_{i=1}^m w_i'f_j(y_i)=\int_\Omega f_j\dint\mu 
		\end{equation*}
        for $j=1,\ldots,n$.
	Hence, $Q_{\mathbf{w}}(X)$ is exact on $V$. In addition, $p_i'\le\max\{p_i,1/k \}=\max\{p_i,1/n\}$. In case  $p_i'\le p_i$, this yields $w_i \le w_i'\le \frac{3}{2m\varrho(y_i)} \le \frac{3}{m} \le \frac{3}{n}$. And also in case $p_i'\le 1/n$, we get
        \[
        w_i= \beta p_i'/\varrho(y_i) \le \frac{3}{n} . \qedhere
        \]
    \end{proof}

\noindent
Assuming that constants are integrated exactly by the rule $Q_{\mathbf{w}}(X)$ in the previous theorem, which, e.g., happens when $\mathbf{1}_\Omega\in V$, we can further record the following observation. 

    \begin{remark}
    If the weights $w_i$ in Theorem~\ref{thm:RT_bounded} sum to $1$, there always exist $\Theta(n)$ weights of size $w_i\asymp 1/n$. For instance, let $I:=\{i:w_i\ge 1/(3n)\}$. Then $|I|\ge n/8$, which follows from
    \[
     1-  \frac{2n-|I|}{3n} \le 1 - \sum_{i\notin I} w_i = \sum_{i\in I} w_i \le \frac{3|I|}{n}.  
    \]
    \end{remark}
        
	\subsection{A ``compressed'' version of  Theorem~\ref{thm:main2}}
	We use the vectors $\boldsymbol{\psi}_n(y)=(e_j(y))_{j>n}$ from~\eqref{eqdef:psi_n}, for which $\|\boldsymbol{\psi}_n(y)\|_2^2=E_n(y)$.
	
	\begin{proposition}
		\label{prop:compress}
		Let $\mu$ be a probability measure, let $n\in\N$, and let $\varrho_n$ be the mixture density~\eqref{mix_density}. Further, let $Q_{\mathbf{w}}(Y)$ be a positive quadrature rule with nodes $Y=(y_1,\ldots,y_m)\in D^m$ and weights $0\le w_i\le 3/(2m\varrho_n(y_i))$, which is exact on $V_n$. Then, for every $k\in\N$, there is a positive quadrature rule $Q_{\mathbf{w}'}(X)$ with nodes $X\subset Y$, $|X|\le k+n$, which is exact on $V_n$ and satisfies
		\begin{equation}
			\label{eq:compress_error}
			\sup_{f\in B(\cH)}
			\biggl|\int_{\Omega} f\dint\mu-Q_{\mathbf{w}'}(X)f\biggr|
			\le
			\sup_{f\in B(\cH)}
			\biggl|\int_{\Omega} f\dint\mu-Q_{\mathbf{w}}(Y)f\biggr|
			+\frac{9}{4}\sqrt{\frac{T_n}{k}\Bigl(1+\frac{n+1}{k}\Bigr)} .
		\end{equation}
	\end{proposition}
	
	\begin{proof}
		We may assume $\mathbf{w}\neq\mathbf{0}$, and set $\beta\coloneqq \sum_{i=1}^m w_i\varrho_n(y_i)\in(0,3/2]$.
		If $n\ge r$, then $E_n\equiv0$ and $T_n=0$.
		Otherwise, since $\mu(\Omega)=1$, the mixture density~\eqref{mix_density} and the inequality of arithmetic and geometric means give $\varrho_n\ge\frac13(1+E_n/T_n)\ge\frac23\sqrt{E_n/T_n}$.
		In both cases, the vectors $\mathbf{u}_i\coloneqq \boldsymbol{\psi}_n(y_i)/\varrho_n(y_i)$, $i=1,\ldots,m$, satisfy
		\begin{equation*}
			%\label{eq:compress_R}
			\|\mathbf{u}_i\|_2=\frac{\sqrt{E_n(y_i)}}{\varrho_n(y_i)}\le \frac32\sqrt{T_n}\eqqcolon R .
		\end{equation*}
		Let $f_1,\ldots,f_{n'}$ be a basis of $V_n$, with $n'\coloneqq\dim V_n\le n$.
		We apply Lemma~\ref{lem:compress} to the vector $\mathbf p$ with $p_i\coloneqq w_i\varrho_n(y_i)/\beta$ and the matrix $\mathbf{L}\in\R^{d\times m}$ with first row $(1,\ldots,1)$ and remaining rows $(f_j(y_i)/\varrho_n(y_i))_{i=1}^m$, $j=1,\ldots,n'$, so that $d\coloneqq n'+1$.
		Based on the resulting $\mathbf{p}'$, we set $w_i'\coloneqq \beta p_i'/\varrho_n(y_i)\ge0$, and let $X$ consist of those nodes $y_i$ with $w_i'\neq0$.
		Then $|X|\le k+n'\le k+n$ and due to $\mathbf{L}\mathbf{p}'=\mathbf{L}\mathbf{p}$, for $j=1,\ldots,n'$, it holds that
		\begin{equation*}
			\sum_{i=1}^m w_i'f_j(y_i)=\beta\sum_{i=1}^m p_i'\frac{f_j(y_i)}{\varrho_n(y_i)}
			=\beta\sum_{i=1}^m p_i\frac{f_j(y_i)}{\varrho_n(y_i)}=\sum_{i=1}^m w_if_j(y_i)=\int_\Omega f_j\dint\mu ,
		\end{equation*}
		namely that $Q_{\mathbf{w}'}(X)$ is exact on $V_n$.
		
		For the error estimate, note that $X\subset Y\subset D$ and that both rules are exact on $V_n$.
		Since $(w_i-w_i')\boldsymbol{\psi}_n(y_i)=\beta(p_i-p_i')\mathbf{u}_i$, the $\ell_2$-identity in Proposition~\ref{prop:err_decomp} and the triangle inequality yield
		\begin{align*}
			\sup_{f\in B(\cH)}\biggl|\int_\Omega f\dint\mu-Q_{\mathbf{w}'}(X)f\biggr|
			&=\biggl\|\int_\Omega\boldsymbol{\psi}_n\dint\mu-\sum_{i=1}^m w_i'\boldsymbol{\psi}_n(y_i)\biggr\|_2\notag\\
			&\le\biggl\|\int_\Omega\boldsymbol{\psi}_n\dint\mu-\sum_{i=1}^m w_i\boldsymbol{\psi}_n(y_i)\biggr\|_2+\biggl\|\sum_{i=1}^m(w_i-w_i')\boldsymbol{\psi}_n(y_i)\biggr\|_2\notag\\
			&=\sup_{f\in B(\cH)}\biggl|\int_\Omega f\dint\mu-Q_{\mathbf{w}}(Y)f\biggr|+\beta \biggl\|\sum_{i=1}^m(p_i-p_i')\mathbf{u}_i\biggr\|_2 .
		\end{align*}
		By Lemma~\ref{lem:compress} and $d\le n+1$, the second summand is at most $\frac32\sqrt{(R^2/k)(1+(n+1)/k)}=\frac{9}{4}\sqrt{(T_n/k)(1+(n+1)/k)}$.
	\end{proof}

	\subsection{Proof of Theorem~\ref{thm:main1}}
	
	For the first part of the construction, let us assume that $\mu$ is a probability measure.
	We apply Corollary~\ref{cor:main2} with $\delta\coloneqq\frac12$ and with $C'\coloneqq400$ in place of its constant $C$ (not to be confused with the constant in Theorem~\ref{thm:main1}), i.e., to $m\coloneqq\lceil 804C'n\log(24n)\rceil$ random nodes.
	As the event happens with positive probability and the randomly drawn nodes lie in $D$ almost surely, there exist nodes $Y=(y_1,\ldots,y_m)\in D^m$ and weights $\mathbf{w}$ satisfying~\eqref{eq:main2_weights} such that $Q_{\mathbf{w}}(Y)$ is exact on $V_n^\prime\supset V_n$ and has worst-case error at most $(1+\sqrt3) \sigma_{n+1}+\sqrt{\log(\ee C')/C'}\sqrt{T_n/n}$.
	Since $m$ only determines the size of the candidate pool, the constant $C'$ does not affect the final number of nodes.
	Hence, Proposition~\ref{prop:compress} applied with $k = n+1$ yields a positive quadrature rule $Q_{\mathbf{w}'}(X)$ with at most $k+n=2n+1$ nodes $X\subset Y$.
	By~\eqref{eq:compress_error} and Corollary~\ref{cor:main2}, its worst-case error on $B(\cH)$ is bounded by
	\begin{equation*}
		\underbrace{(1+\sqrt3) \sigma_{n+1}+\sqrt{\frac{\log(\ee C')}{C'}}\sqrt{\frac{T_n}{n}}}_{\text{Corollary~\ref{cor:main2}}}
		+\frac{9}{2\sqrt2}\sqrt{\frac{T_n}{n}}
		\le(1+\sqrt3) \sigma_{n+1}+A\sqrt{\frac{T_n}{n}},
	\end{equation*}
	where $A\coloneqq \sqrt{\log(400\ee)/400}+\frac{9}{2\sqrt2}\le 3.315$.
	Let $C_2\coloneqq19$. As there are at least $(1-\frac1{C_2})n$ indices $j$ with $\frac{n}{C_2}\le j\le n$, we have \smash{$(1-\frac1{C_2})\sigma_{n+1}^2\le\frac1n\sum_{n/C_2\le j\le n}\sigma_j^2$}.
	With the Cauchy-Schwarz inequality 
    \begin{equation*}
	    (ax+by)^2\le\Bigl(\frac{a^2}{1-1/C_2}+b^2\Bigr)\biggl(\Bigl(1-\frac1{C_2}\Bigr)x^2+y^2\biggr),
	\end{equation*}
    and since the index ranges $n/C_2\le j\le n$ and $j>n$ are disjoint, this gives
	\begin{equation}\label{eq:kappa3n}
		\kappa_{2n+1}^+(\cH,\mu)^2
		\le\Bigl(\frac{(1+\sqrt3)^2}{1-1/C_2}+A^2\Bigr)\biggl(\Bigl(1-\frac1{C_2}\Bigr)\sigma_{n+1}^2+\frac{1}{n}\sum_{j>n}\sigma_j^2\biggr)
		\le\frac{C_2}{n}\sum_{j\ge n/C_2}\sigma_j^2,
	\end{equation}
	where we used $\frac{19}{18}(1+\sqrt3)^2+A^2\le18.9\le C_2$.
	
	For $0<\mu(\Omega)<\infty$, we apply this construction to the probability measure $\tilde{\mu}\coloneqq \mu/\mu(\Omega)$, whose singular numbers are $\sigma_j/\sqrt{\mu(\Omega)}$, and use $\kappa_m^+(\cH,\mu)=\mu(\Omega) \kappa_m^+(\cH,\tilde\mu)$.
	This gives an additional factor $\mu(\Omega)$ on the right-hand side, and the case $\mu(\Omega)=0$ is trivial.
	It remains to bring this bound into the form of Theorem~\ref{thm:main1}, i.e., to remove the constant $C_2$ and to start the sum at $n$ instead of $n/C_2$.
	To this end, we trade the constant for a larger number of nodes and apply~\eqref{eq:kappa3n} with $C_2n$ in place of $n$.
	Since the factor $C_2$ cancels, we get
	\begin{equation*}
		\kappa_{2C_2n+1}^+(\cH,\mu)^2\le\frac{\mu(\Omega)}{n}\sum_{j\ge n}\sigma_j^2 .
	\end{equation*}
	With $C_3\coloneqq2C_2+1=39$, the monotonicity of $\kappa_m^+(\cH,\mu)$ in $m$ proves Theorem~\ref{thm:main1}.
	\hfill\qed
	%\end{proof}
	
	\begin{remark}
\label{rem:nonconstructive}
Lemma~\ref{lem:compress} can also be combined with Richter's theorem~\cite[Thm.~4]{Richter57} instead of Theorem~\ref{thm:main2}.
Let $\phi\coloneqq(1+\sqrt5)/2$ denote the golden ratio, so that $1+\phi=\phi^2=(3+\sqrt5)/2\approx2.62$.
This yields a non-constructive variant of Theorem~\ref{thm:main1} with constant $C=\phi^2$, namely
\begin{equation}
\label{eq:nonconstructive}
\kappa^{+}_{m}(\cH,\mu)\le\biggl(\frac{\mu(\Omega)}{n}\sum_{j\ge n}\sigma_j^2\biggr)^{1/2}\qquad\text{for all } n\in\N \text{ and } m\ge\phi^2\cdot n .
\end{equation}
As in the proof of Theorem~\ref{thm:main1}, it suffices to consider a probability measure $\mu$.
We show~\eqref{eq:nonconstructive} with $n+1$ in place of $n$ for $n\in\N_0$, i.e., $\kappa^{+}_{m}(\cH,\mu)^2\le T_n/(n+1)$ for all $m\ge\phi^2\cdot(n+1)$.
If $n\ge r$, then $E_n\equiv0$ and $T_n=0$, and we set $\varrho\coloneqq1$.
Otherwise, we set $\varrho\coloneqq\frac12(1+E_n/T_n)$, and the inequality of arithmetic and geometric means gives \smash{$\varrho\ge\sqrt{E_n/T_n}$}.
In both cases, $\varrho\ge\frac12$ is a probability density, and the vectors $\mathbf{u}_n(x)\coloneqq \boldsymbol{\psi}_n(x)/\varrho(x)$ satisfy $\|\mathbf{u}_n(x)\|_2=\sqrt{E_n(x)}/\varrho(x)\le\sqrt{T_n}\eqqcolon R$.
Let $f_1,\ldots,f_{n'}$ be a basis of $V_n$, where $n'\coloneqq\dim V_n\le n$, and let $\mathbf{P}_N$ denote the orthogonal projection onto the first $N$ coordinates of $\ell_2(\N)$.
The functions $1$ and $f_j/\varrho$, $j=1,\ldots,n'$, are integrable with respect to $\varrho\mu$, since $\int_\Omega|f_j|\dint\mu<\infty$.
Moreover, since $\int_\Omega\|\mathbf{u}_n\|_2^2 \varrho\dint\mu=\int_\Omega E_n/\varrho\dint\mu\le 2T_n$, the coordinates of $\mathbf{P}_N\mathbf{u}_n$ and the function $\|(\mathbf{I}-\mathbf{P}_N)\mathbf{u}_n\|_2^2$ are integrable with respect to $\varrho\mu$ as well.
As $\mu(\Omega\setminus D)=0$, Richter's theorem applied to $D$ equipped with $\varrho\mu$ yields nodes $x_1,\ldots,x_m\in D$ and a vector $\mathbf{p}=(p_1,\ldots,p_m)$ with $p_i\ge0$ and $\sum_{i=1}^m p_i=1$ such that $\sum_{i=1}^m p_i g(x_i)=\int_\Omega g \varrho\dint\mu$ for each of these functions $g$.
Hence, the rule with weights $w_i\coloneqq p_i/\varrho(x_i)$ is exact on $V_n$, and $w_i\boldsymbol{\psi}_n(x_i)=p_i\mathbf{u}_n(x_i)$.
Since, in addition, $\sum_{i=1}^m p_i\mathbf{P}_N\mathbf{u}_n(x_i)=\mathbf{P}_N\int_\Omega\boldsymbol{\psi}_n\dint\mu$, it follows from~\eqref{eq:err_identity} and the Cauchy-Schwarz inequality with respect to $\mathbf{p}$ that its worst-case error satisfies
\begin{equation*}
\biggl\|\int_\Omega\boldsymbol{\psi}_n\dint\mu-\sum_{i=1}^m p_i\mathbf{u}_n(x_i)\biggr\|_2
\le\biggl(\int_\Omega\|(\mathbf{I}-\mathbf{P}_N)\mathbf{u}_n\|_2^2 \varrho\dint\mu\biggr)^{1/2}+\biggl\|(\mathbf{I}-\mathbf{P}_N)\int_\Omega\boldsymbol{\psi}_n\dint\mu\biggr\|_2\eqqcolon\varepsilon_N .
\end{equation*}
Here, $\varepsilon_N\to0$ as $N\to\infty$, by dominated convergence with majorant $\|\mathbf{u}_n\|_2^2 \varrho$ for the first term, and since $\int_\Omega\boldsymbol{\psi}_n\dint\mu\in\ell_2(\N)$ for the second.
Proceeding as in the proof of Proposition~\ref{prop:compress} with $\beta=1$, $\varrho$ in place of $\varrho_n$, $R=\sqrt{T_n}$, and $k\coloneqq\lceil\phi\cdot (n+1)\rceil$, we obtain a positive rule which is exact on $V_n$ and has at most $k+n'<\phi\cdot(n+1)+1+n=\phi^2\cdot(n+1)$ nodes in $D$.
Its worst-case error is at most $\varepsilon_N+\sqrt{(R^2/k)(1+(n'+1)/k)}\le\varepsilon_N+\sqrt{T_n/(n+1)}$, where we used $n'\le n$ and the fact that $k\ge\phi\cdot(n+1)$ implies $k^2\ge(n+1)(k+n+1)$, as $\phi$ is the positive root of $t^2-t-1$.
Since $N$ is arbitrary, the monotonicity of the numbers $\kappa_m^+(\cH,\mu)$ in $m$ gives $\kappa_m^+(\cH,\mu)^2\le T_n/(n+1)$ for all $m\ge\phi^2\cdot(n+1)$, which is~\eqref{eq:nonconstructive} with $n+1$ in place of $n$, since $T_n=\sum_{j\ge n+1}\sigma_j^2$.

In contrast to this existence argument, the route via Theorem~\ref{thm:main2} provides an explicit sampling distribution, independent nodes, and a computable pool of $O(n\log(\ee n))$ candidate nodes with controlled weights, from which the final rule is obtained by the finite procedure of Lemma~\ref{lem:compress}.
\end{remark}

	\noindent\textbf{Acknowledgment.} TU is supported by the German Research Foundation (DFG) with grant Ul403/4-1. Several of the presented relations were developed with the assistance of LLMs by Anthropic and OpenAI. Additional verification, presentation and writing were carefully carried out by the authors, who take full responsibility for the results. 
    
    %\bibliography{literature_10}
	%\bibliographystyle{abbrv}

	\begin{appendix}
		\section{Tools from probability}
        %\label{sec:ProbTools}
		
		The relevant tools from probability theory  used in this article are collected in this appendix. 
		The first theorem is a classical contraction principle~\cite[Thm.~4.12]{LedouxTalagrand}, comparing the expectation of suprema of related stochastic processes.
		
		\begin{theorem}[Contraction principle]
			\label{thmA:Contraction}
			Let $F\colon \R_+\to\R_+$ be convex and increasing, and let
			$\varphi_i\colon\R\to\R$, $i=1,\ldots,n$, satisfy $\varphi_i(0)=0$ and $|\varphi_i(s)-\varphi_i(t)|\le |s-t|$ for all $s,t\in\R$.
			If $\boldsymbol{\epsilon}\coloneqq (\epsilon_1,\ldots,\epsilon_n)$ are independent Rademacher
			variables, then, for every bounded set $T\subset\R^n$, it holds
			\begin{equation*}
				\E_{\boldsymbol{\epsilon}}\biggl[
				F\Big(
				\frac12
				\sup_{t\in T}
				\Big|
				\sum_{i=1}^n
				\epsilon_i\varphi_i(t_i)
				\Big|
				\Big)\biggr]
				\le
				\E_{\boldsymbol{\epsilon}}
				\biggl[F\Big(
				\sup_{t\in T}
				\Big|
				\sum_{i=1}^n
				\epsilon_i t_i
				\Big|
				\Big)\biggr].
			\end{equation*}
		\end{theorem}
		
		\noindent
		The next result is the symmetrization principle~\cite[Lem.~8.4]{FoRa13}, see also~\cite[Lem.~6.3]{LedouxTalagrand}.
		
		\begin{theorem}[Symmetrization]
			\label{thmA:Symm}
			Let $\mathbf{u}_i$, $i=1,\ldots,n$, be independent random vectors in a finite-dimensional vector space endowed with a norm $\|\cdot\|$, and let $\mathbf{x}_i\coloneqq\E\mathbf{u}_i$.
			If $\boldsymbol{\epsilon}\coloneqq(\epsilon_1,\ldots,\epsilon_n)$ are independent Rademacher variables, independent of $\{\mathbf{u}_i\}_{i\le n}$, then, for $1\le p<\infty$,
			\begin{equation*}
				\E_{\mathbf{u}}\Bigl\|\sum_{i=1}^n(\mathbf{u}_i-\mathbf{x}_i)\Bigr\|^p\le2^p \E_{\mathbf{u},\boldsymbol\epsilon}\Bigl\|\sum_{i=1}^n\epsilon_i\mathbf{u}_i\Bigr\|^p .
			\end{equation*}
		\end{theorem}
		
		\noindent
		Relations of tail bounds to moment estimates are standard, see e.g.~\cite[Prop.~7.13]{FoRa13} and~\cite[Cor.~4.3]{MoStUl24}. A converse of~\cite[Prop.~7.13]{FoRa13} is provided by~\cite[Prop.~7.11]{FoRa13}, which we recite below.
		
		\begin{proposition}[Moments and tails]
			\label{propA:MomTail}
			Suppose that $Z$ is a random variable satisfying
			\begin{equation*}
				\bigl(\E|Z|^p\bigr)^{1/p}\le\alpha\beta^{1/p}p^{1/\nu}\qquad\text{for all }p\in[p_0,p_1]
			\end{equation*}
			for some constants $\alpha,\beta,\nu,p_1>p_0>0$. Then
			\begin{equation*}
				\Prob\bigl(|Z|\ge{\ee}^{1/\nu}\alpha u\bigr)\le\beta {\ee}^{-u^\nu/\nu}\qquad\text{for all }u\in\bigl[p_0^{1/\nu},p_1^{1/\nu}\bigr].
			\end{equation*}
		\end{proposition}
		
		%\subsection{Concentration results}
		
		\noindent
		We proceed with concentration results for sums of bounded random variables.
		The following bound for the moment generating function is the key step in the proof of Bernstein's inequality~\cite[Cor.~7.31]{FoRa13}, see the proof of~\cite[Thm.~7.30]{FoRa13} together with~\cite[(7.42)]{FoRa13}.
		
		\begin{lemma}[Bernstein's moment generating function bound]
			\label{lemA:BernMGF}
			Let $\{X_i\}_{i\le n}$ be independent real-valued centered random variables such that $|X_i|\le L$ almost surely and $\E X_i^2\le\sigma_i^2$ for $i=1,\ldots,n$, with constants $L,\sigma_i>0$, and let $\sigma^2\coloneqq\sum_{i=1}^n\sigma_i^2$.
			Then, for all $0<\theta<3/L$,
			\begin{equation*}
				\E\exp\Bigl(\theta\sum_{i=1}^nX_i\Bigr)\le\exp\Bigl(\frac{\theta^2\sigma^2}{2(1-L\theta/3)}\Bigr).
			\end{equation*}
		\end{lemma}
		
		\noindent
		From Lemma~\ref{lemA:BernMGF}, we derive the following moment estimate.
		
		\begin{lemma}
			[Moment estimation]
			\label{lem:bernmom}
			Let $\{X_i\}_{i\le n}$ be independent real-valued centered random variables with $|X_i|\le L$ almost surely, for each index $i=1,\ldots,n$, and $\sum_{i=1}^n\E X_i^2\le v$, where $L,v>0$.
			Then, for every $p\ge1$ with $L\sqrt{p/v}\le\frac18$,
			\begin{equation*}
            %\label{eq:bernmom}
				\E\Bigl|\sum_{i=1}^nX_i\Bigr|^p\le 2\bigl({\ee}^{-11/23}\sqrt{pv}\bigr)^p .
			\end{equation*}
		\end{lemma}
		\begin{proof}
			Let $Z\coloneqq\sum_{i=1}^nX_i$ and $0<\theta<3/L$.
			Lemma~\ref{lemA:BernMGF}, applied to $\{X_i\}_{i\le n}$ and $\{-X_i\}_{i\le n}$ with $\sigma_i^2\coloneqq\E X_i^2+\eta/n$, and $\eta\downarrow0$, gives $\E \exp(\pm\theta Z)\le\exp\bigl(\frac{\theta^2v}{2(1-L\theta/3)}\bigr)$.
			Since $\max_{t\ge0}t^p\ee^{-\theta t}=(p/({\ee}\theta))^p$, we obtain
			\begin{equation*}
				\E|Z|^p\le\Bigl(\frac{p}{\ee\theta}\Bigr)^p \E{\ee}^{\theta|Z|}\le\Bigl(\frac{p}{\ee\theta}\Bigr)^p\bigl(\E{\ee}^{\theta Z}+\E{\ee}^{-\theta Z}\bigr)\le 2\Bigl(\frac{p}{\ee\theta}\Bigr)^p\exp\Bigl(\frac{\theta^2v}{2(1-L\theta/3)}\Bigr).
			\end{equation*}
			We choose $\theta\coloneqq\sqrt{p/v}$, for which $L\theta/3\le\frac1{24}$ by assumption, so that the exponent is at most $\frac{12p}{23}$.
			This gives $\E|Z|^p\le2(\sqrt{pv}/{\ee})^p{\ee}^{12p/23}=2({\ee}^{-11/23}\sqrt{pv})^p$.
		\end{proof}
		
		\noindent
		The final theorem deals with infinite-dimensional matrices and was proved by Moeller and Ullrich, see~\cite[Prop.~3.8]{MoUl20}. Recall our convention that elements of $\ell_2(\N)$ are treated as column vectors.
		
		\begin{theorem}[Infinite matrix Chernoff inequality]
			\label{thm:MoUl20}
			Let $\mathbf{u}_i$, $i=1,\ldots,n$, be independent identically distributed random sequences from \(\ell_2(\N)\).
			Let further $n \geq 3$, $s > 1$, $L>0$ such that \(\|\mathbf{u}_i\|_{2} \leq L\) almost surely and \(\E(\mathbf{u}_i \mathbf{u}_i^\ast)  =  \mathbf{\Lambda}\)  for all \(i=1,\ldots,n\).
			Then, it holds
			\begin{align*}
				%\label{quant:F}
				\Prob \Big( \Big\| \frac{1}{n} \sum_{i=1}^n   \mathbf{u}_i  \mathbf{u}_i^\ast - \mathbf{\Lambda} \Big\|_{2 \to 2} \geq F \Big) \leq 2^\frac{3}{4}  n^{1-s} ,
			\end{align*}
			where 
			\[F \coloneqq  \max\biggl\{ \frac{8s \log n}{n} L^2 \bigl(\tfrac{1+\sqrt{5}}{2}\bigr)^2 , \| \mathbf{\Lambda} \|_{2\to2}\biggr\}.
			\]
		\end{theorem}
		
	\end{appendix}
\end{document}